\documentclass[11pt]{article}
\usepackage{amssymb,amsmath,accents}
\usepackage{amscd}
\usepackage{amsfonts,amsthm,mathrsfs}
\usepackage{setspace}
\usepackage{cases,empheq}
\usepackage{subcaption}
\usepackage{diagbox}
\usepackage[hyphens,allowmove]{url}
\usepackage{hyperref}
\usepackage{cleveref}
\usepackage{float}

\numberwithin{equation}{section}
\newtheorem{thm}{Theorem}
\newtheorem{cor}[thm]{Corollary}
\newtheorem{lem}[thm]{Lemma}
\newtheorem{prop}[thm]{Proposition}

\newtheorem{rem}[thm]{Remark}

\begin{document}

\title
{On dynamic stability of energetically stable equilibria \\ of the Navier-Stokes-Korteweg flows \vspace{0.1em}\\ %
\vspace{0.4em}
\normalsize{In the memory of Professor Hermann Sohr 
-- our distinguished colleague}}

\author{Yoshikazu Giga$^1$, Naoto Kajiwara$^2$ and Kazuyuki Tsuda$^3$ \vspace{0.7em}\\ 
\centerline{{\small $^1$Graduate School of Mathematical Sciences}} \\
\centerline{{\small The University of Tokyo}} \\
\centerline{{\small 3-8-1 Komaba, Meguro-ku, Tokyo 153-8914, Japan}} \vspace{0.7em}\\
\centerline{{\small $^2$Applied Physics Course, Department of Electrical, Electronic and Computer Engineering}} \\
\centerline{{\small Gifu University}} \\
\centerline{{\small 1-1 Yanagido, Gifu City, Gifu 501-1193, Japan}} \vspace{0.7em}\\
\centerline{{\small $^3$Center for Fundamental Education}} \\
\centerline{{\small Kyushu Sangyo University}} \\
\centerline{{\small 2-3-1 Matsukadai, Higashi-ku, Fukuoka 813-8503, Japan}}}

\date{}

\maketitle
\thispagestyle{empty}
%

\begin{abstract}
We consider the Navier-Stokes-Korteweg equations in a bounded domain or a periodic cell.
 The pressure considered in this paper may not be monotone with respect to the density so that there exist non-constant equilibria allowing two-phases.
 Using a simple Hilbert space framework, we prove that if an isolated equilibrium is energetically stable and non-degenerate, it is exponentially stable under the isothermal Navier-Stokes-Korteweg flows when the space dimension is less than or equal to three.
 For non-isolated case, we prove that a global-in-time solution near an energetically stable equilibrium converges to possibly another equilibrium exponentially fast.
 No smallness assumptions on equilibria are imposed.
 For the proof we apply a (generalized) stability principle due to J.~Pr\"{u}ss, M.~Wilke and G.~Simonett (2013).
\end{abstract}

\maketitle

\section{Introduction} \label{SI}

In this paper, we consider the (isothermal) Navier-Stokes-Korteweg equations in $\Omega$ which is a smooth bounded domain in $\mathbb{R}^n$ or a periodic cell $\mathbb{T}^n=\prod_{i=1}^n(\mathbb{R}/\omega_i\mathbb{Z})$ with $\omega_i>0$ ($i=1,\ldots,n$).
 It is of the form
\begin{gather}
	\rho_t + \operatorname{div}(\rho u) = 0
	\quad\text{in}\quad \Omega\times(0,T) \label{ECM} \\
	\rho \left(u_t + (u\cdot\nabla) u\right)
	- \mu\Delta u
	- (\lambda + \mu) \nabla \operatorname{div} u + \nabla p(\rho)
	= \rho \nabla (\varepsilon\Delta\rho)
	\quad\text{in}\quad \Omega\times(0,T). \label{ECMM}
\end{gather}
Here $\rho$ and $u=(u^1,\ldots,u^n)$ denotes the unknown density and velocity, respectively.
 The constants $\mu$ and $\lambda$ are shear and bulk viscosity, respectively.
 When we say that \eqref{ECM}--\eqref{ECMM} are the Navier-Stokes-Korteweg equations, we always assume that $\mu>0$ and $\lambda+2\mu>0$.
 (If the non-isothermal system satisfies the second law of the thermo-dynamics, $\mu$ and $\lambda$ must satisfy $\mu\ge0$ and $n\lambda+2\mu\ge0$.)
 We set the Lam\'e operator by
\begin{equation} \label{EL}
	Lu := -\mu\Delta u - (\lambda + \mu) \nabla \operatorname{div} u
\end{equation}
for later convenience.
 The condition $\mu>0$ and $2\mu+\lambda>0$ corresponds to the strong ellipticity of $L$.
 The Korteweg relaxation parameter is denoted by $\varepsilon$ and it is assumed to be a positive constant.
 The pressure  $p$ is written by using a given available (Helmholtz) energy $\Psi\in C^2(\mathbb{R})$ as
\[
	p(\rho) = \rho \Psi_\rho(\rho) - \Psi(\rho), \quad
	\Psi_\rho(\rho) = \partial\Psi/\partial\rho
\]
so that
\begin{equation} \label{EEP}
	\nabla p(\rho) = \rho \nabla \left(\Psi_\rho(\rho) \right).
\end{equation}
We are interested in the case that $\Psi$ may not be convex or equivalently $p$ is not monotone increasing.
 A typical example of $\Psi$ is
\begin{equation} \label{ESp}
	\Psi(\rho) = W(\rho) + c_1\rho
\end{equation}
with a double-well potential $W$ with equal depth and a positive constant $c_1$.
 Physically speaking, the small minimum point of $W$ corresponds to the saturated density $\rho_g$ of gas while the large minimum point of $W$ corresponds to the saturated density $\rho_\ell$ of liquid; see Figure~\ref{FE}.
\begin{figure}[h]
    \centering
    \includegraphics[width=0.35\linewidth]{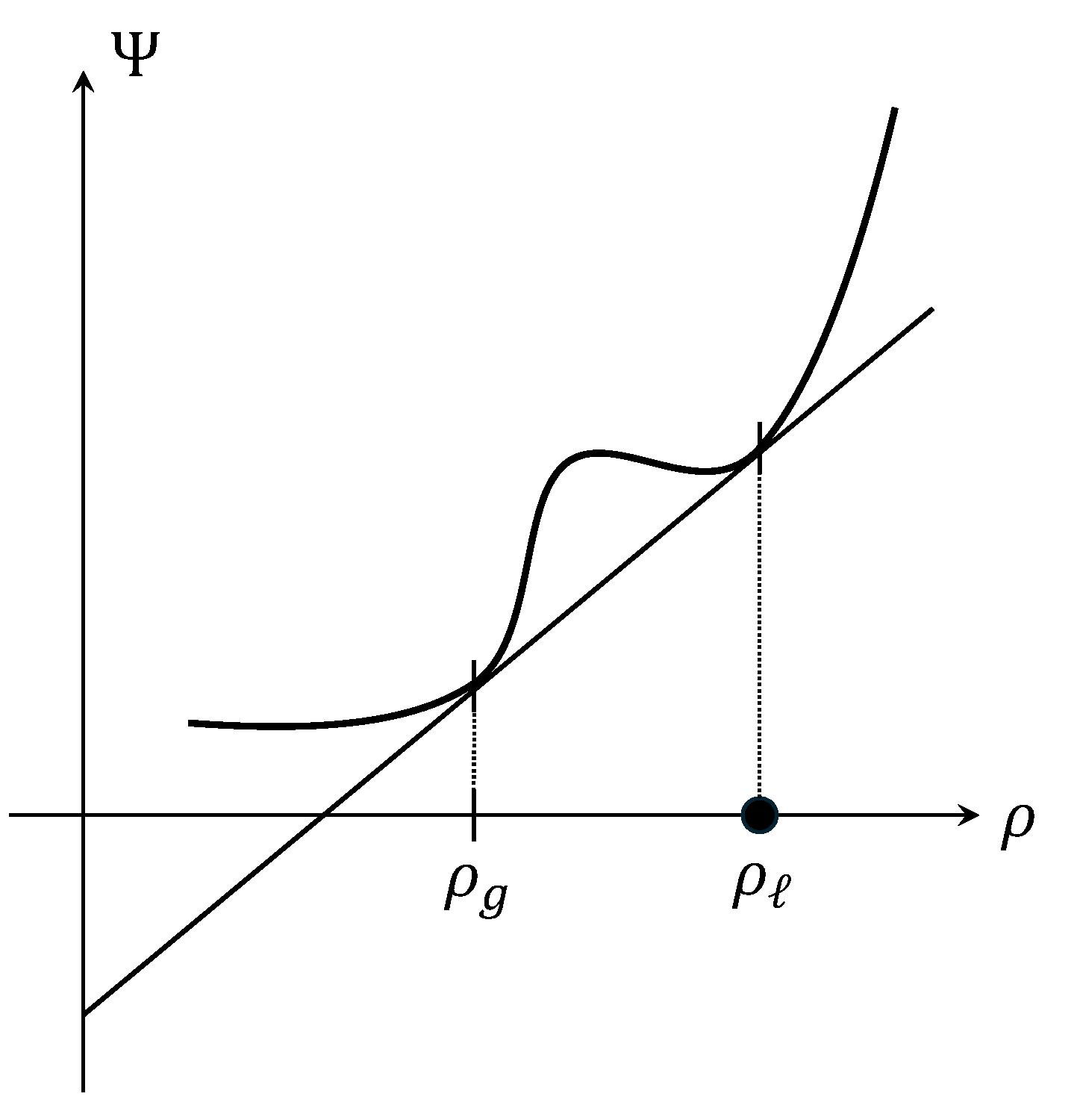}
    \caption{A typical $\Psi$}
    \label{FE}
\end{figure}
Using \eqref{EEP}, the equation \eqref{ECMM} becomes
\begin{equation} \label{ECMM2}
	u_t + (u\cdot\nabla)u + \rho^{-1} Lu
	+ \nabla \left(\Psi_\rho(\rho) - \varepsilon\Delta\rho \right) = 0
\end{equation}
provided that $\rho>0$, where $L$ is defined by \eqref{EL}.
 We consider \eqref{ECM} and \eqref{ECMM2} supplemented by the boundary condition
\begin{equation} \label{EB}
	\frac{\partial\rho}{\partial\nu} = 0, \quad
	u = 0 \quad\text{on}\quad \partial\Omega \times (0,T),
\end{equation}
where $\nu$ denotes the unit exterior normal of $\partial\Omega$.
 In the case $\Omega=\mathbb{T}^n$, we interpret that a periodic boundary condition for a cell $\prod_{i=1}^n[0,\omega_i)$ is imposed.

We are interested in the stability of equilibria of \eqref{ECM}, \eqref{ECMM2} (with \eqref{EB}).
 We say that $(\rho,u)$ is an \emph{equilibrium} (or \emph{stationary solution}) if $\rho$ and $u$ are time-independent.
 It turns out that the velocity $u$ must be a constant if $(\rho,u)$ is an equilibrium provided that we assume $\mu>0$ and $\lambda+2\mu>0$.
 If $\partial\Omega\neq\emptyset$, then by \eqref{EB} $u$ must be zero.
 For a moment we shall discuss the case $\partial\Omega\neq\emptyset$.
 If $u=0$, then $(\rho,0)$ is an equilibrium if and only if $\rho$ solves
\[
    \nabla \left(\Psi_\rho(\rho) - \varepsilon\Delta\rho \right) = 0
    \quad\text{in}\quad \Omega, \quad
    \frac{\partial\rho}{\partial\nu} = 0
    \quad\text{on}\quad \partial\Omega.
\]
We denote such $\rho$ by $\rho^e$.
 Let $\mathcal{E}$ denote the set of all $\rho^e$'s.
 We say that $\rho^e$ (or $(\rho^e,0)$) is \emph{energetically stable} if
\begin{equation} \label{ESt}
	\int_\Omega \left\{ \varepsilon|\nabla y|^2
	+ \Psi_{\rho\rho} (\rho^e) y^2 \right\}\, dx \ge 0
\end{equation}
for all $y\in H^1(\Omega)$ satisfying $\int_\Omega y\,dx=0$.
 As we shall see in Section \ref{SEq}, there always exists an energetically stable equilibrium $\rho^e$ for fixed average, i.e.
\begin{equation} \label{Eav}
	\frac{1}{|\Omega|} \int_\Omega \rho^e\,dx = \rho_\mathrm{av}
\end{equation}
for a given constant $\rho_\mathrm{av}$.
 If $-\Psi_{\rho\rho}(\rho_\mathrm{av})/\varepsilon$ is sufficiently large, an energetically stable equilibrium $\rho^e$ satisfying the average condition \eqref{Eav} should not be a constant function.
 Indeed, a constant satisfying \eqref{Eav} must be $\rho_\mathrm{av}$ itself.
 Let $\alpha$ be the minimal positive eigenvalue of the Neumann Laplacian.
 In other words,
\[
	\alpha = \inf \left\{ \frac{\int_\Omega |\nabla y|^2\,dx}{\int_\Omega |y|^2\,dx} \biggm|
	y \in H^1(\Omega),\ \int_\Omega y\,dx = 0 \right\}.
\]
Since $\alpha>0$, the stability condition \eqref{ESt} for $\rho^e\equiv\rho_\mathrm{av}$ fails if and only if $-\Psi_{\rho\rho}(\rho_\mathrm{av})>\varepsilon\alpha$.
In other words, $-\Psi_{\rho\rho}(\rho_\mathrm{av})\le\varepsilon\alpha$ is a necessary and sufficient condition so that an energetically stable equilibrium satisfying \eqref{Eav} is a constant function.

For $\rho^e\in\mathcal{E}$ let $B(\rho^e)$ be the linearized operator defined by
\[
	B(\rho^e) = P_1 \left(-\varepsilon\Delta + \Psi_{\rho\rho}(\rho^e) \right)
	\quad\text{with}\quad
	P_1 g = g - \frac1{|\Omega|} \int_\Omega g\,dx
\]
in $H_\mathrm{av}^1(\Omega)$ where $H_\mathrm{av}^1(\Omega)$ is the average-free $H^1$ space.
 We say that $\rho^e$ is \emph{non-degenerate} if $\mathcal{E}$ is an $m$-dimensional manifold near $\rho^e$ and the dimension of the kernel $\ker B(\rho^e)$ equals $m$.
 If $m=0$ so that $\rho^e$ is isolated, then $\rho^e$ is non-degenerate if and only if $B(\rho^e)$ is injective.
 In the case of $\Psi$ in Figure~\ref{FE}, $\Psi_{\rho\rho}(\rho_\mathrm{av})<0$ for some $\rho_\mathrm{av}\in(\rho_g,\rho_\ell)$.
 Thus for sufficiently small $\varepsilon>0$, the constant function $\rho\equiv\rho_\mathrm{av}$ is not energetically stable.
 When $\Omega$ is a bounded interval, it turns out that there exists a monotone energetically stable equilibrium for small $\varepsilon>0$; see Figure~\ref{FM}.
\begin{figure}[h]
    \centering
    \includegraphics[width=0.45\linewidth]{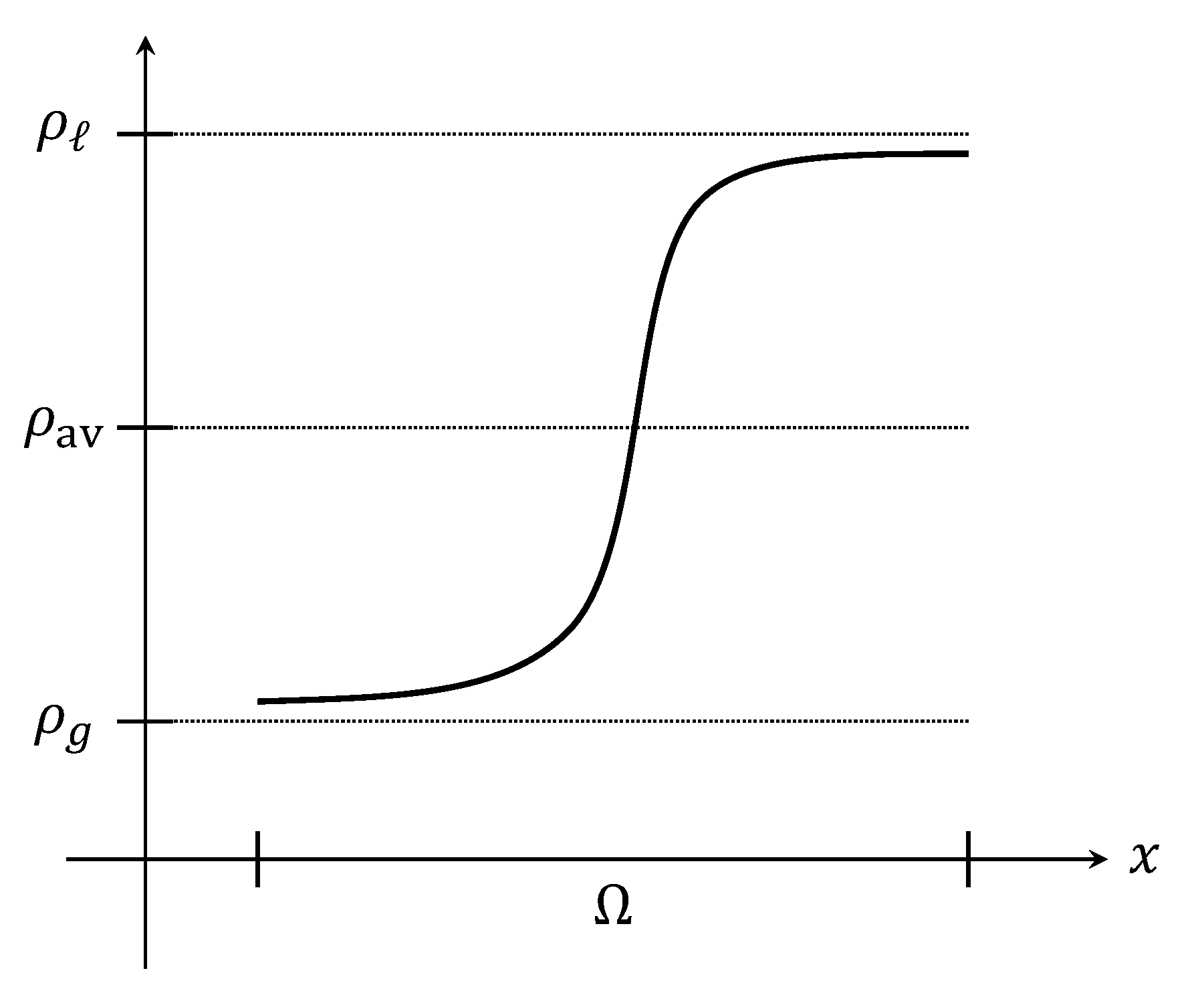}
    \caption{An energetically stable equilibrium}
    \label{FM}
\end{figure}

In the case of periodic boundary condition, if $(\rho,u)$ is an equilibrium, we only conclude that $u$ is a constant.
 However, we may assume that $u=0$ by using a moving frame with the constant speed $u$; see e.g.\ \cite{GKT}.
 The definitions of energetic stability and non-degeneracy are the same as in the case of $\partial\Omega\neq\emptyset$.
 Moreover, an energetically stable equilibrium satisfying the constraint \eqref{Eav} is a constant function $\rho_\mathrm{av}$ if and only if $-\Psi_{\rho\rho}(\rho_\mathrm{av})\le\varepsilon\alpha$ as in the case of $\partial\Omega\neq\emptyset$.
 However, the dimension $m=\ker B(\rho^e)\ge1$.
 Indeed, if $\Omega=\mathbb{T}^n$, $\rho^e(\cdot+c)$ is also an equilibrium for any $c\in\mathbb{R}^n$.
 Since
\[
    c\cdot\nabla\left(\Psi_\rho(\rho^e)-\varepsilon\Delta\rho^e\right)=0
\]
implies that $c\cdot\nabla\rho^e\in\ker B(\rho^e)$, we observe that $m\ge1$ unless $\rho^e$ is a constant.
 We say that $c\in\mathbb{R}^n\setminus\{0\}$ is a constant direction of $\rho^e$ if $c\cdot\nabla\rho^e\equiv0$.
 For simplicity, we shall often assume that there is no constant direction of $\rho^e$, i.e., $c\cdot\nabla\rho^e\not\equiv0$ for any $c\in\mathbb{R}^n\setminus\{0\}$.

We shall discuss its dynamic stability in a simple Hilbert space framework to avoid technical complexity.
 Our main results are summarized as follows.
\begin{thm}[isolated case] \label{TM1}
Assume that $1\le n\le3$.
 Assume that $\mu>0$ and $2\mu+\lambda>0$.
 Let $\Omega$ be a smooth bounded domain in $\mathbb{R}^n$.
 Assume that $\Psi\in C^3$.
 Let $(\rho^e,0)$ (with $\rho^e\in C^3(\bar{\Omega})$) be a positive energetically stable equilibrium of \eqref{ECM}, \eqref{ECMM2} with \eqref{EB}.
 Assume that $\rho^e$ is non-degenerate and isolated.
 Then, there exists small $\eta>0$ such that if the initial value $\rho_0=\left.\rho\right|_{t=0}$, $\left.u_0=u\right|_{t=0}$ satisfies
\[
	\| \rho_0 - \rho^e \|_{H^2}
	+ \| u_0 \|_{H^1} \le \eta, \quad
	\int_\Omega \rho_0\, dx = \int_\Omega \rho^e\,dx, \quad
	\left. u_0 \right|_{\partial\Omega} = 0, \quad
	\frac{\partial\rho_0}{\partial\nu} = 0 \ \text{on}\ \partial\Omega,
\]
then there exists a unique global-in-time solution $(\rho,u)$ of \eqref{ECM}, \eqref{ECMM2} and \eqref{EB} with initial data $(\rho_0,u_0)$.
 Moreover,
\[
	\left\| \rho(t) - \rho^e \right\|_{H^2}
	+ \| u(t) \|_{H^1} \le Ce^{-\kappa t}
	\quad\text{for}\quad t>0,
\]
with some $\kappa>0$ and $C$.
\end{thm}
\begin{thm}[general case] \label{TM2}
Assume that $1\le n\le3$ and that $\Psi\in C^3$.
 Assume that $\mu>0$ and $2\mu+\lambda>0$.
 Let $\Omega$ be a smooth bounded domain in $\mathbb{R}^n$.
 Let $(\rho^e,0)$ (with $\rho^e\in C^3(\bar{\Omega})$) be a positive energetically stable equilibrium of \eqref{ECM}, \eqref{ECMM2} with \eqref{EB}.
 Assume that $\rho^e$ is non-degenerate.
 Then there exists (small) $\eta>0$ such that if the initial data $\rho_0$, $u_0$ satisfying
\[
	\| \rho_0 - \rho^e \|_{H^2} + \| u_0 \|_{H^1} \le \eta, \quad
	\int_\Omega \rho_0\, dx = \int_\Omega \rho^e\, dx, \quad
    \left. u_0 \right|_{\partial\Omega}=0, \quad
    \frac{\partial\rho_0}{\partial\nu}=0
    \ \text{on}\ \partial\Omega,
\]
there exists a unique global-in-time solution $(\rho,u)$ of \eqref{ECM}, \eqref{ECMM2} and \eqref{EB} with initial data $(\rho_0,u_0)$.
 Moreover, there is $\rho_{**}\in\mathcal{E}$ such that
\[
	\left\| \rho(t) - \rho_{**} \right\|_{H^2}
	+ \left\| u(t) \right\|_{H^1}
	\le Ce^{-\kappa t}
\]
with some $\kappa>0$ and $C>0$.
 The conclusion is still valid when $\Omega=\mathbb{T}^n$ provided that there is no constant direction of $\rho^e$ (under no boundary conditions for $u_0$, $\rho_0$ and $u$, $\rho$.)
\end{thm}
\begin{rem} \label{RMain}
In the case of $\Omega=\mathbb{T}^n$, even if we do not assume that there is no constant direction of $\rho^e$, we conclude the same result provided that the term $\left\|u(t)\right\|_{H^1}$ in the last formula is replaced by
\[
    \left\|u(t)-c\right\|_{H^1}
\]
for some $c\in\mathbb{R}^n$ which is a constant direction of $\rho^e$.
\end{rem}
Note that we do not assume any smallness condition on $\rho^e$ nor $\nabla\rho^e$.

The dynamic stability result of Theorem~\ref{TM2} is consistent with numerical experiment for $\Omega=\mathbb{T}^1$ in \cite{GKT}, where an equilibrium having one minimum and one maximum looks dynamically stable.
 In \cite{GKT}, $\Psi$ is taken as a double-well potential.

To analyze the system \eqref{ECM}, \eqref{ECMM2} with \eqref{EB}, we rewrite them into an abstract evolution equation.
 If we introduce a linear operator $P(U_*)$ depending on $U_*=(\rho_*,u_*)^T$ defined by
\begin{align*}
	P(U_*)U &= \left(
	\begin{array}{cc}
	\operatorname{div}(\rho_* u) \\
	-\nabla(\varepsilon\Delta\rho) + \displaystyle\frac{1}{\rho_*} Lu + (u_* \cdot \nabla)u
	\end{array}
	\right), \quad
	U = \dbinom{\rho}{u} \\
	&= \left(
	\begin{array}{cc} 
   0 & \operatorname{div}(\rho_*\cdot) \\ 
   -\nabla(\varepsilon\Delta) & \displaystyle\frac{L}{\rho_*} + (u_* \cdot \nabla) 
  \end{array}
	\right)
	\dbinom{\rho}{u},
\end{align*}
then \eqref{ECM}, \eqref{ECMM2} can be written as
\[
	U_t + P(U)U = F_0(U), \quad
	F_0(U) = \dbinom{0}{-\nabla\Psi_\rho(\rho)}.
\]
We notice that if we impose \eqref{EB}, then the total mass $\int_\Omega\rho\,dx$ of the solution of \eqref{ECM} and \eqref{ECMM2} is conserved.
 In other words, the total mass is time-independent.
 Indeed,
\[
	\frac{d}{dt} \int_\Omega \rho(x,t)\, dx
	= -\int_\Omega \operatorname{div}(\rho u)\, dx
	= -\int_{\partial\Omega} \nu\cdot\rho u\, d\mathcal{H}^{n-1} = 0
\]
by the condition $u=0$ on $\partial\Omega$, where $\mathcal{H}^{n-1}$ denotes the $n-1$ dimensional Hausdorff measure.
 It turns out that it is more convenient to write the system for $\tilde{\rho}=\rho-\rho^e$ and $u$ since $\int_\Omega\tilde{\rho}_0\, dx=0$ implies $\int_\Omega\tilde{\rho}\, dx=0$ for all time.
 The system for $\tilde{U}=(\tilde{\rho},u)^T$ can be written
\[
	\tilde{U}_t + \tilde{P}(\tilde{U})\tilde{U} = \tilde{F}(\tilde{U})
\]
with
\begin{align*}
	&\tilde{P}(\tilde{U}) = \left(
	\begin{array}{cc} 
   0 & \operatorname{div}\left( (\rho^e+\tilde{\rho})\cdot \right) \\ 
   -\nabla(\varepsilon\Delta) & \displaystyle\frac{L}{\rho^e+\tilde{\rho}} + (u \cdot \nabla)
   \end{array}
	\right), \\
	&\tilde{F}(\tilde{U}) = \dbinom{0}{-\nabla\Psi_\rho(\rho^e+\tilde{\rho}) + \nabla\Psi_\rho(\rho^e)}	
\end{align*}
since $\nabla\varepsilon\Delta\rho^e=\nabla\Psi_\rho(\rho^e)$.
 To write \eqref{ECM}, \eqref{ECMM2} with \eqref{EB} we fix function spaces and the domain of the operator $\tilde{P}$ to include the boundary condition.
 The basic space we take is
\[
	X_0 = H_\mathrm{av}^1(\Omega) \times \left(L^2(\Omega)\right)^n, \quad
	H_\mathrm{av}^1(\Omega) = \left\{ \rho \in H^1(\Omega) \biggm| \int_\Omega\rho\,dx = 0 \right\}.
\]
As a solution space we take
\[
	X_1 = \left\{ \rho \in H^3(\Omega) \cap H_\mathrm{av}^1(\Omega) \biggm|
	\frac{\partial\rho}{\partial\nu} = 0\ \text{on}\ \partial\Omega \right\} 
	\times \left\{ u \in \left( H^2(\Omega) \right)^n \Bigm|
	u = 0\ \text{on}\ \partial\Omega \right\}.
\]
Then the real interpolation space $(X_0,X_1)_{1/2,2}$ agrees with the complex interpolation space $[X_0,X_1]_{1/2}$ since $X_i$ ($i=1,2$) are Hilbert spaces
 (see e.g.\ \cite[Corollary~4.37]{L2}).
 In this particular case,
\begin{align*}
	X_{1/2} := (X_0,X_1)_{1/2,2} 
	&= \left\{ \rho \in H^2(\Omega) \cap H_\mathrm{av}^1(\Omega) \biggm|
	\frac{\partial\rho}{\partial\nu} = 0\ \text{on}\ \partial\Omega \right\} \\
	&\times \left\{ u \in \left( H^1(\Omega) \right)^n \Bigm|
	u = 0\ \text{on}\ \partial\Omega \right\}.
\end{align*}
For small $U=(\rho,u)^T$ in $X_{1/2}$, we set
\[
	A(U) : X_1 \to X_0
\]
such that
\[
	A(U) U_1 = \tilde{P}(U)U_1 \quad\text{with}\quad U = (\rho,u)^T, \quad
	U_1 = (\rho_1, u_1)^T \in X_1.
\]
This is well defined for $1\le n\le3$ since the Sobolev embedding $\|\rho\|_\infty\le C\|\rho\|_{H^2}$ for $\int_\Omega\rho\,dx=0$ holds in these dimensions (Proposition \ref{PBa}).
 We set
\[
	F(U) := \tilde{F}(U)
\]
for $U=(\rho,u)^T \in X_{1/2}$.
 Then the equation \eqref{ECM} \eqref{ECMM2} with \eqref{EB} can be written as
\begin{equation} \label{EAb}
	\frac{dU}{dt} + A(U)U = F(U)
\end{equation}
where $U=(\tilde{\rho},u)^T$.

The basic strategy to prove stability is to apply generalized stability principle due to J.\ Pr\"{u}ss, M.~Wilke and G.\ Simonett \cite{PWS} in the form of \cite[Theorem~5.3.1]{PS}.
 The major assumptions one has to check are spectral properties of linearized operator $A_0$ of $A(U)U-F(U)$ at $U=0$.
 In the case of an isolated equilibrium, we shall prove that all spectra of $A_0$ have a positive real part by using energetical stability and non-degeneracy condition.
 A key identity is
\begin{equation} \label{Ekey}
	\xi \int_\Omega B(\rho^e) \bar{\rho}\rho\, dx 
	+ \bar{\xi} \int_\Omega \rho^e |u|^2\, dx
	+ \int_\Omega u\cdot L \bar{u}\, dx = 0
\end{equation}
for $(\xi+A_0)U=0$, where $B(\rho^e)=-\varepsilon\Delta+\Psi_{\rho\rho}(\rho^e)$, which is a linearized operator of $-\varepsilon\Delta\rho+\Psi_\rho(\rho)$ around $\rho^e$;
 here $\bar{\rho}$ denotes the complex conjugate of $\rho$.
 Energetically stability condition says that $\int_\Omega B(\rho^e)\bar{\rho}\rho\,dx\ge0$.
 Since $\mu$ and $\lambda$ satisfy $\mu>0$ and $\lambda+2\mu>0$ so that $\int_\Omega\bar{u}\cdot Lu\,dx\ge0$, we see that $\operatorname{Re}\xi\le0$, $\operatorname{Im}\xi=0$.
 Non-degeneracy condition implies that $0$ is not an eigenvalue.

In the case of the periodic boundary condition, the analysis is more involved because equilibria (stationary solutions) consist of a manifold of positive dimension.
 In particular, $0$ is an eigenvalue.
 Even in this case \cite[Theorem~5.3.1]{PS} is applicable.
 Our identity implies that all spectra are non-negative.
 Our non-degeneracy condition implies that $\rho^e$ is linearly normally stable in the sense of J.\ Pr\"{u}ss and G.\ Simonett \cite[Theorem~5.3.1]{PS}.
 For this purpose, we shall show that $0$ is a semi-simple eigenvalue of $A_0$ and that the dimension of $\ker A_0$ equals the dimension of the manifold of equilibria, of course in the case of a bounded domain, this approach is still valid.

For a linearized operator, we further need maximum regularity.
 This has been proved by \cite{Ko} even when the base space $X_0$ is $L^p$ type under the ellipticity assumption $\mu>0$ and $\lambda+2\mu>0$.
 In \cite[Theorem~2.1]{Ko}, $\lambda$, $\mu$ and $\varepsilon$ are allowed to depend on the time and the space variables to get the ($L^p$ in space and time) maximum regularity result.
 We also note that a general theory of a parabolic mixed order system based on Newton polygons yields the maximal regularity with different exponents in space and time for the linearized operator of constant coefficients in $\mathbb{R}^n$ as explained in \cite[\S4.4]{DK}.
 In our Hilbert space case, it suffices to prove the sectoriality of $A_0$ with angle less than $\pi/2$ to get maximum regularity (see e.g.\ \cite[Proposition~2.3]{Sa}).
 However, it turns out that it is difficult to prove based on bilinear form associated with $A_0$ since $A_0$ is far from self-adjoint.
 There is a general theory for operators consists of blocks like $A_0$ (see Subsection~\ref{SSAE}) by \cite{AH}.
 However, our operator $A_0$ is not diagonal dominated so the general theory of \cite{AH} does not apply.
 In the case of one dimensional setting $n=1$, it is easy to prove the sectoriality so we give its proof.

Since the maximum regularity in particular implies the sectoriality, $(\zeta+A_0)^{-1}$ exists for sufficiently large $\zeta>0$.
 Moreover, $X_1$ is compactly embedded in $X_0$, it turns out that $(\zeta+A_0)^{-1}$ is compact operator in $X_0$ for large $\zeta>0$.
 Thus, all spectra of $A_0$ consists of eigenvalues with no accumulation points in $\mathbb{C}$.
 By \eqref{Ekey}, we conclude that all spectra of $A_0$ has a positive real part except $0$ (Theorem~\ref{TSp}).

Although dynamic stability like Theorem~\ref{TM1} and Theorem~\ref{TM2} should be a fundamental topic, such a problem was not well studied when $\Omega$ is a bounded domain or $\mathbb{T}^n$ especially when $\Psi$ is not convex.
 In \cite{Ko2}, M.~Kotschote proved the dynamic stability of equilibrium when $\Omega$ is a bounded domain under the boundary condition \eqref{EB}.
 In his case $\Psi$ is allowed to be non-convex.
 A potential type external force is also allowed.
 However, if we consider the case where there is no external force, it is assumed that $\Psi_{\rho\rho}+\varepsilon\alpha$ is always assumed to be positive where $\alpha$ is the minimal positive eigenvalue of the Neumann Laplacian.
 This assumption implies that the energy
\[
	E(\rho) = \int_\Omega  \left\{ \frac\varepsilon2 |\nabla\rho|^2 + \Psi(\rho) \right\} dx
\]
is strictly convex, so there exists at most one critical point which is a minimizer; see Section \ref{SEq}.
 Since an equilibrium is a critical point of $E$ with constraint \eqref{Eav} and since $\rho=\rho_\mathrm{av}$ itself is always an equilibrium, there is only one equilibrium $\rho=\rho_\mathrm{av}$.
 Thus, an equilibrium in \cite{Ko2} must be a constant if there is no external force.
 Our stability results apply for not only such a trivial equilibrium but also for a non-constant equilibrium provided that it is energetically stable and non-degenerate.
 In \cite{Ko2}, stability of equilibrium with constant temperature is actually discussed for the non-isothermal Navier-Stokes-Korteweg equations in $L^p$ setting.
 We note that our dynamic stability result extends to the non-isothermal case but in the framework of Hilbert space it applies only to one dimensional setting, i.e.\ $n=1$.
 We need $L^p$ maximal regularity result as in \cite{Ko}, \cite{Ko2}.
 We shall discuss such an extension in our forthcoming paper.
 In this paper we restrict ourselves into isothermal case to clarify the structure of the problem by avoiding technical complexity since the paper \cite{Ko2} is very involved because the setting is quite general, for example, $\lambda$, $\mu$ and $\varepsilon$ may not be constants.

Although there are many papers on mathematical analysis for the Navier-Stokes-Korteweg equations, the study of the boundary value problem is quite limited.
 M.~Kotschote \cite{Ko} constructed a local smooth solution both for bounded and exterior domains for isothermal version.
 It is extended by \cite{Ko1} for the non-isothermal case.
 T.~Kobayashi, M.~Murata and H.~Saito \cite{KMS} derived $L^p$ maximal regularity for bounded and exterior domain and discuss the stability of a constant equilibrium.

Concerning derivation of NSK, Van der Waals \cite{vW, vW2} observes that a phase transition boundary caused by a steep gradient of the density, which is called as a diffuse interface.
 Based on his idea, Korteweg \cite{Kor}, who was a student of him, gives  the stress tensor including the term $\nabla\rho \otimes \nabla \rho$ of the Navier-Stokes equation.
 Dunn and Serrin \cite{DS} generalize the Korteweg's work and provide the system \eqref{ECM}--\eqref{ECMM} with the Korteweg tensor.
 More recently, Heida and M\'{a}lek \cite{HM} also formulate \eqref{ECM}--\eqref{ECMM} by maximizing the entropy production; see also a nice review article by M\'alek and Pr\r{u}\v{s}a \cite{MP}.
 Freist\"{u}hler and Kotschote \cite{FK1, FK2} derive the Navier-Stokes-Allen-Cahn system and the Navier-Stokes-Chan-Hilliard system which describe two phase flow of mixture materials from some model of Korteweg type.
 Note that the equation derived in \cite{FK1} is different from \cite{DS} and \cite{FK2} for non-isothermal case.
 Gorban and Karlin \cite{GK} derive the Korteweg tensor from the Boltzmann equation.

Concerning solvability in $\mathbb{R}^n$, especially, stability of a constant state, Danchin and Desjardins \cite{DD} show global existence of a solution around the motionless state $(\rho_*, 0)$ with a small initial data $u_0 \in (B^{\frac{n}{2}}_{2,1}\cap B^{\frac{n}{2}-1}_{2,1})\times B^{\frac{n}{2}-1}_{2,1}$, where $B^{\frac{n}{2}}_{2,1}$ denotes the usual homogeneous Besov space. 
Hattori and Li \cite{HL1, HL2} show the global existence of a $H^{N+1} \times H^{N}$ solution  with a small 
$u_0 \in H^{N+1}\times H^N$, where $N$ is an integer satisfying that $N\geq [n/2]+2$ and $[n/2]$ denotes the integer part of $n/2$. We also cite results of Tan and R. Zhang, X, Zhang and Tan \cite{TZ, ZT}  Tan, Wang and Xu \cite{TWX}, and Wang and Tan \cite{WT} including decay rates for $L^2$ class. 
Murata and Shibata \cite{MS} study an $L^p$-in-time and $L^q$-in-space setting by the maximal regularity.  In recent studies, Chikami and Kobayashi \cite{CK}, as well as Kobayashi and Murata \cite{KM} and Kobayashi and Tsuda \cite{Kobayashi-Tsuda} prove the asymptotic stability of the above steady state satisfying the critical condition, {\rm i.e., } $P'(\rho_*) = 0$ including decay rates for the whole space problem.

For the Navier-Stokes-Cahn-Hilliard equation, a similar dynamic stability is proved by \cite{A}, where the \L ojasiewicz-Simon inequality is involved.
 This approach does not require non-degeneracy but it needs the analyticity of $\Psi$.

This paper is organized as follows.
 In Section~\ref{SEq}, we discuss structure of the equilibrium set.
 In Section~\ref{SL}, we give a rigorous definition of operators.
 We also discuss the kernel of $A_0$.
 We recall the maximum regularity result.
 In Section~\ref{SA}, we apply the abstract stability principle due to \cite{PS}.
 Since the paper is dedicated to late Professor Hermann Sohr, in Subsection~\ref{SSH} we also mention one of his seminal and by now classical result for the Stokes operator when we prove that $A_{00}$ is invertible in the case that $\Omega$ is a bounded domain.

\section{Equilibrium} \label{SEq}

In this section, we prove that the velocity is constant in equilibrium state.
 We notice that it is highly non-trivial to show its non-degeneracy although it is rather easy to show the existence of an energetically stable equilibrium.

We first recall the energy identity.
\begin{prop} \label{PEE}
Let $\Omega$ be a smooth bounded domain in $\mathbb{R}^n$ or $\Omega=\mathbb{T}^n$ ($n\ge1$).
 Let $(\rho,u)$ be a smooth solution to \eqref{ECM}, \eqref{ECMM}.
 Then
\begin{equation} \label{EEid}
    \frac{d}{dt} \int_\Omega \left( \rho \frac{|u|^2}{2} + \Psi(\rho) + \frac{\varepsilon|\nabla\rho|^2}{2} \right)\, dx
    = -\int_\Omega Lu\cdot u\, dx
\end{equation}
provided that $(\rho,u)$ satisfies $u\cdot\nu=0$, $\partial\rho/\partial\nu=0$ on the boundary $\partial\Omega$ (if it exists).
\end{prop}
\begin{proof}
    This follows just integration by parts.
 We calculate
\[
    \frac{d}{dt} \int_\Omega \rho \frac{|u|^2}{2}\, dx 
    = \int_\Omega \rho_t \frac{|u|^2}{2}\, dx
    + \int_\Omega \rho uu_t\, dx.
\]
Using \eqref{ECM}, we observe that
\begin{align*}
    \int_\Omega \rho_t \frac{|u|^2}{2}\, dx
    &= \int_\Omega \left( -\operatorname{div}(\rho u) \right) \frac{|u|^2}{2}\, dx \\
    &= \int_\Omega \rho u(u\cdot\nabla)u\, dx
\end{align*}
by integration by parts and $u\cdot\nu=0$ on $\partial\Omega$.
 We thus observe that
\begin{equation} \label{EK1}
    \frac{d}{dt} \int_\Omega \rho \frac{|u|^2}{2}\, dx
    = -\int_\Omega Lu\cdot u\, dx
    - \int_\Omega \rho u \nabla\Psi_\rho\, dx
    + \int_\Omega \rho u \nabla(\varepsilon\Delta\rho)\, dx
\end{equation}
by \eqref{ECMM}.
 Using \eqref{ECM}, we observe that
\begin{equation}
\begin{aligned} \label{EInt}
    \frac{d}{dt} \int_\Omega \left( \Psi(\rho) + \frac{\varepsilon|\nabla\rho|^2}{2} \right)\, dx
    &= \int_\Omega (\Psi_\rho \rho_t + \varepsilon\nabla\rho \cdot \nabla\rho_t)\, dx \\
    &= \int_\Omega \left( -\Psi_\rho \operatorname{div}(\rho u) + \varepsilon\Delta\rho(\operatorname{div}\rho u) \right)\, dx
\end{aligned}
\end{equation}
by integration by parts with $\partial\rho/\partial\nu=0$ on $\partial\Omega$.
 By \eqref{EK1} and \eqref{EInt}, integrating by parts with $u\cdot\nu=0$ on $\partial\Omega$, we obtain \eqref{EEid}.
\end{proof}

Note that this identity \eqref{EEid} still holds if we replace
\[
    \Psi(\rho) + \frac{\varepsilon|\nabla\rho|^2}{2}
    \quad\text{by}\quad \Psi(\rho,\nabla\rho)
\]
provided that $p=\rho\Psi_\rho-\Psi$ and $\rho\nabla(\varepsilon\Delta\rho)$ in \eqref{ECMM} is replaced by $\rho\nabla\operatorname{div}\Psi_{\nabla\rho}$.
 In particular, $\varepsilon$ is allowed to depend on $\rho$ if we replace $\rho\nabla(\varepsilon\Delta\rho)$ by $\rho\nabla\left(\operatorname{div}(\varepsilon\nabla\rho)\right)$.
 We emphasize that there is no ``dissipation" term in \eqref{EEid} due to the Korteweg term $\rho\nabla\operatorname{div}\Psi_{\nabla\rho}$.

We shall estimate the dissipation $\int_\Omega Lu\cdot u\,dx$ in \eqref{EEid} due to viscosities.
 In the following, we assume that $\Omega$ is a smooth bounded domain or $\mathbb{T}^n$ if no assumptions on $\Omega$ are stated.

Let $L_0$ be the Lam\'e operator defined in $\left(L^2(\Omega)\right)^n$ by
\begin{align}
	L_0 u &= -\mu\Delta u - (\mu+\lambda) \nabla(\operatorname{div}u) \label{ELame1} \\
	&= -\operatorname{div} S, \label{ELame2} 
\end{align}
where
\[
	S = 2\mu D(u) + \lambda \operatorname{div} uI, \quad
	D(u) = (\nabla u+\nabla u^T)/2
\]
with
\[
	D(L_0) = \left\{ u \in \left(H^2(\Omega)\right)^n \Bigm|
	u = 0 \ \text{on}\ \partial\Omega \right\}.
\]
We recall a well-known properties of $L_0$.
 For $n\times n$ matrices $X$ and $Y$, the inner product is defined as
\[
	X:Y = \operatorname{tr}(XY^T)
\]
where $\operatorname{tr}$ denotes the trace.
 The corresponding norm is defined
\[
	|X| = (X:X)^{1/2}.
\]
\begin{lem} \label{LLame}
For $u\in D(L_0)$,
\begin{equation} \label{ELH1}
	\int_\Omega u \cdot L_0 u\, dx
	= \int_\Omega \mu |\nabla u|^2 + (\mu+\lambda) |\operatorname{div}u|^2\, dx.
\end{equation}
If $\mu>0$ and $2\mu+\lambda\ge\delta$ with $\delta\in(0,\mu)$, then
\[
	\int_\Omega u\cdot L_0 u\, dx
	\ge \delta \int_\Omega |\nabla u|^2\, dx.
\]
If $\int_\Omega u\cdot L_0u\, dx=0$, then $u$ is constant and if moreover, $\partial\Omega\neq\emptyset$, then $u=0$.
\end{lem}
\begin{proof}
We use \eqref{ELame1} for $\int_\Omega u\cdot L_0u\,dx$.
 Integrating by parts yields \eqref{ELH1}.
 Integrating by parts twice, we obtain
\begin{align*}
    \int_\Omega |\operatorname{div}u|^2\, dx
	&= \sum_{i,j=1}^n \int_\Omega u_{x_j}^j u_{x_i}^i \, dx \\
    &= -\sum_{i,j=1}^n \int_\Omega u^j u_{x_i x_j}^i \, dx
    + \int_{\partial\Omega} u^j \nu_j u_{x_i}^i \, d\mathcal{H}^{n-1} \\
    &= \sum_{i,j=1}^n \int_\Omega u_{x_i}^j u_{x_j}^i \, dx 
    - \int_{\partial\Omega} u^j \nu_i u_{x_j}^i \, d\mathcal{H}^{n-1} + 0 \\
    &= \sum_{i,j=1}^n \int_\Omega u_{x_i}^j u_{x_j}^i \, dx,
\end{align*}
where $f_{x_i}=\partial f/\partial x_i$, $f_{x_i x_j}=\partial^2f/(\partial x_i\partial x_j)$ for $f=f(x_1,\ldots,x_n)$.
 Since
\[
	\sum_{i,j=1}^n u_{x_i}^j u_{x_j}^i
	= \nabla u : \nabla u^T \le |\nabla u|^2, 
\]
\eqref{ELH1} yields
\begin{align*}
	\int_\Omega \left\{ \mu |\nabla u|^2\, dx
	+ (\mu+\lambda) |\operatorname{div}u|^2 \right\}
	& = \int_\Omega \left\{ \delta |\nabla u|^2 + (\mu-\delta) |\nabla u|^2
	+ (\mu+\lambda) (\operatorname{div}u)^2 \right\} dx \\
	&\ge \int_\Omega \left\{ \delta |\nabla u|^2 + (\mu-\delta) (\operatorname{div}u)^2
	+ (\mu+\lambda) (\operatorname{div}u)^2 \right\} dx.
\end{align*}
Since $2\mu+\lambda\ge\delta$, this yields (i\hspace{-0.1em}i).
 (This estimate is often called the Korn inequality.)
\end{proof}

The energy identity together with Lemma \ref{LLame} implies that the velocity is constant.
\begin{cor} \label{CEq}
    Assume that $\mu>0$ and $2\mu+\lambda>0$.
 Let $(\rho,u)$ ($\rho>0$) be a time-independent solution (equilibrium) to \eqref{ECM} and \eqref{ECMM} with \eqref{EB} if $\partial\Omega\neq\emptyset$.
 Then $u$ is a constant.
 Thus $\rho$ must satisfy $\nabla(\Psi_\rho-\varepsilon\Delta\rho)=0$ in $\Omega$.
 If $\partial\Omega\neq\emptyset$, then $u=0$.
\end{cor}
It is easy to find an energetically stable equilibrium.
 However, its non-degeneracy is in general difficult to check.

We consider
\begin{equation} \label{E4St}
	\nabla \left( \Psi_\rho(\rho) - \varepsilon\Delta\rho \right)
	= 0 \quad\text{in}\quad \Omega.
\end{equation}
If $\Omega$ has a boundary, we impose the Neumann condition
\begin{equation} \label{E4B}
	\frac{\partial\rho}{\partial\nu}
	= 0 \quad\text{on}\quad \partial\Omega.
\end{equation}
We also impose the constrained average condition
\begin{equation} \label{E4Av}
	\frac{1}{|\Omega|} \int_\Omega \rho\, dx
	= \rho_\mathrm{av},
\end{equation}
where $\rho_\mathrm{av}\in\mathbb{R}$ is a given constant.
\begin{thm} \label{T4Ex}
Assume that $\Omega$ is a smooth bounded domain in $\mathbb{R}^n$ or $\Omega=\mathbb{T}^n$.
 Assume that $\Psi\in C^2$ is bounded from below on $\mathbb{R}$ up to linear function, i.e., $\Psi(\rho)+c_1\rho$ is bounded from below with some $c_1$.
 Assume that $|\Psi_{\rho\rho}|$ is bounded on $\mathbb{R}$.
 Then there exists an energetic stable equilibrium $\rho^e\in C^{2+\sigma}(\bar{\Omega})$ for $\sigma\in(0,1)$ for \eqref{E4St} and \eqref{E4B} satisfying \eqref{E4Av}.
 If $\Psi\in C^3$, then $\rho^e\in C^{3+\sigma}(\bar{\Omega})$.
\end{thm}
\begin{proof}
This is rather standard.
 We only give a sketch of the proof.
 We may assume that $\Psi\ge0$ by adding a linear function and a constant.
 We consider a functional
\[
	E(\rho) = \int_\Omega \left\{ \frac\varepsilon2 |\nabla\rho|^2 + \Psi(\rho) \right\} dx.
\]
We seek a minimizer in a space
\[
	Y = \left\{ \rho \in H^1(\Omega) \biggm|
	\frac1{|\Omega|} \int_\Omega \rho\, dx = \rho_\mathrm{av} \right\}.
\]

By a direct method we see that there exists a minimizer $\rho^e$ of $E$ in $Y$.
 Indeed, let $\{\rho_m\}\subset Y$ a minimizing sequence, i.e., $E(\rho_m)\to\inf_Y E$ ($<\infty$) as $m\to\infty$.
 Then $\{\rho_m\}$ is bounded in $H^1$ since $\Psi\ge0$.
 By the Rellich compactness theorem, $\{\rho_m\}$ has a strong convergent subsequence in $L^2(\Omega)$.
 We further take a subsequence so that $\{\rho_m\}$ converges to its limit $\rho^e$ almost everywhere.
 We may assume that $\rho_m$ converges to $\rho^e$ weakly in $H^1$ by weak compactness of a bounded set in a Hilbert space.
 Thus
\[
	\int_\Omega \Psi(\rho^e)\, dx
	\le \varliminf_{m\to\infty} \int_\Omega \Psi(\rho_m)\, dx
\]
by Fatou's lemma and the lower semicontinuity of norms under weak topology implies that
\[
	\int_\Omega |\nabla\rho^e|^2 \, dx
	\le \varliminf_{m\to\infty} \int_\Omega |\nabla\rho_m|^2\, dx.
\]
This implies the lower semicontinuity:
$E(\rho^e)\le\varliminf E(\rho_m)=\inf_Y E$.
 We thus conclude that $\rho^e$ is a minimizer of $E$ on $Y$.

Since $\rho^e$ is a critical point,
\begin{align*}
	\left. \frac{d}{d\delta}  E (\rho^e + \delta y) \right|_{\delta=0}
	&= \left. \int_\Omega \left\{ \varepsilon \nabla \rho^e \cdot \nabla y 
    + \varepsilon\delta|\nabla y|^2
	+ \Psi_\rho (\rho^e+\delta y) y \right\} dx \right|_{\delta=0} \\
	&= \int_\Omega \left\{\varepsilon\nabla \rho^e \cdot \nabla y + \Psi_\rho(\rho^e) y \right\} dx = 0
\end{align*}
for all $y\in H_\mathrm{av}^1$.
 In other words, $\rho^e$ is a weak solution of the Euler-Lagrange equation
\begin{equation} \label{E4S}
	-\varepsilon\Delta\rho + \Psi_\rho(\rho) = \gamma \quad\text{in}\quad \Omega, \quad
	\partial\rho/\partial\nu = 0 \quad\text{on}\quad \partial\Omega, 	
\end{equation}
where $\gamma\in\mathbb{R}$ is some constant.
 A bootstrap argument implies that $\rho^e\in C^{2+\sigma}(\bar{\Omega})$.
 Indeed, by the Sobolev embedding (see e.g.\ \cite{E}), $\rho^e\in H^1=W^{1,2}$ implies that $\rho^e\in L^q$ with $1/q\ge1/2-1/n$, $q<\infty$.
 Since $|\Psi_{\rho\rho}|$ is bounded so that $\left|\Psi_\rho(\rho)\right|\le c\left(|\rho|+1\right)$ with some $c>0$, by the $L^q$ elliptic regularity theory (see e.g.\ \cite{E}, \cite{GT}), we see that $\rho^e\in W^{2,q}(\Omega)$.
 By the Sobolev embedding $\rho^e\in L^{q_1}$, $1/q_1 \ge 1/q-2/n$, $q_1<\infty$. This implies $\Psi_\rho(\rho^e)\in L^{q_1}$.
 Again by the $L^q$ elliptic regularity theory, $\rho^e\in W^{2,q_1}$.
 Repeating this procedure, we observe that $\rho^e\in W^{2,r}$ for $1/r-2/n<0$.
 By the Sobolev embedding (Morrey's inequality), $\rho^e\in C^\sigma(\bar{\Omega})$ with $\sigma=2/n-1/r$.
 We now apply Schauder's elliptic regularity theory to get $\rho^e\in C^{2+\sigma}(\bar{\Omega})$ since $\Psi\in C^2$ with $\rho^e\in C^\sigma$ implies that $\Psi_\rho(\rho^e)\in C^\sigma(\bar{\Omega})$.
 Thus $\rho^e$ solves \eqref{E4S} in a classical sense and it is an equilibrium.
 If $\Psi\in C^3$, $\rho^e\in C^{3+\sigma}(\bar{\Omega})$ by Schauder's elliptic regularity theory.

It remains to prove that $\rho^e$ is energetically stable.
 This is easy since $\rho^e$ is a minimizer of $E$ in $Y$.
 By the minimality,
\[
	\left. \frac{d^2}{d\delta^2} E(\rho^e+\delta y) \right|_{\delta=0} \ge 0.
\]
By a direct calculation, we have
\[
	\frac{d^2}{d\delta^2} E(\rho^e+\delta y) 
	= \int_\Omega \left\{\varepsilon|\nabla y|^2 + \Psi_{\rho\rho}(\rho^e+\delta y) y^2 \right\} dx.
\]
Thus, $\rho^e$ is energetically stable.
\end{proof}
\begin{rem}[non-constancy] \label{RNC}
As mentioned in Section~\ref{SI}, an energetically stable equilibrium is not a constant if $-\Psi_{\rho\rho}(\rho_\mathrm{av})>\varepsilon\alpha$, where $\alpha$ is the minimal positive eigenvalue of the Neumann Laplacian ($B_0$ in Section \ref{SSFO}).
 Thus if $\Psi_{\rho\rho}(\rho_\mathrm{av})<0$, for sufficiently small $\varepsilon>0$, the minimizer constructed in the proof of Theorem \ref{T4Ex} is not a constant.
 If $\Psi$ is taken as
 \[
    \Psi(\rho) = W(\rho) + c\rho
 \]
 with $W(\rho)=(\rho^2-1)^2/4$, $c\in\mathbb{R}$, then $\Psi_{\rho\rho}(\rho_\mathrm{av})<0$ provided that $\rho_\mathrm{av}^2<1/3$ since $W''(\rho)=3\rho^2-1$.
 If $\Omega$ is a bounded interval, it is easy to see that the minimizer must be a monotone function provided that $\varepsilon$ is small for such $\Psi$; see e.g.\ \cite{GKT}, where $W$ is more general. 
\end{rem}

In general it is difficult to claim that these energetically stable equilibrium are non-degenerate even in one-dimensional setting i.e.\ $n=1$.
 In the case of the Neumann boundary condition, the non-degeneracy condition for $\rho^e$ is equivalent to saying that if $y$ satisfies
\[
	-\varepsilon\Delta y + W_{\rho\rho}(\rho^e)y = \tilde{\gamma} \quad\text{in}\quad \Omega, \quad
	\partial y/\partial\nu = 0 \quad\text{on}\quad \partial\Omega
\]
with
\[
	\int_\Omega y\, dx = 0
\]
with some $\tilde{\gamma}\in\mathbb{R}$ must be identically equal to zero.

In one-dimensional setting, it turns out (see \cite{GMT}) that if $\Omega$ is a bounded interval, the minimizer of $E$ in $Y$ is non-degenerate provided that $\varepsilon>0$ is small.
 Note that minimizer must be monotone and isolated.
 This can be proved by adjusting the idea of R.~Schaaf \cite{Sch} using what is called time map.
 The detail will be discussed in our forthcoming paper \cite{GMT}.
 In a periodic setting, the set of minimizers forms a one-parameter family parametrized by shifts.
 The non-degeneracy is equivalent to saying that $y$ must be translation of $\rho_x^e$.
 In \cite[Theorem~2.13]{RW}, it is proved that any periodic solution of \eqref{E4St} with $W(\rho)=\frac14(\rho^2-1)^2$, $\varepsilon=1$ is non-degenerate, although it is stated in a different way.
 They write a solution of
\[
    y'' - W_{\rho\rho} (\rho^e)y = \tilde{\gamma}
\]
by using a variation of constant method.
 However, it seems that their proof needs more explanation to achieve the goal even if $\tilde{\gamma}$ is known to be zero.
 In the third paragraph of \cite[p.~5837]{RW}, it is claimed that some function must vanish somewhere but this is not clear. 

\section{Linearized operator} \label{SL}

\subsection{Function spaces and operators} \label{SSFO}

To study large time behavior of a solution of \eqref{EAb}, we shall study basic properties of $A$ and $F$.

We first give a rigorous definition of $A$ and $F$.
 Let $B_0$ be the Laplacian with the Neumann boundary condition in $L_\mathrm{av}^2(\Omega)$, where $L_\mathrm{av}^2(\Omega)$ denotes the space of average-free square-integrable functions i.e.
\[
	L_\mathrm{av}^2(\Omega)
	= \left\{ \rho \in L^2(\Omega) \biggm|
	\int_\Omega \rho\, dx = 0 \right\}.
\]
 In other words, we define
\[
	B_0 \rho = -\Delta \rho
\]
for $\rho\in D(B_0)$ with
\[
	D(B_0) = \left\{ \rho \in H^2(\Omega)
	\cap L_\mathrm{av}^2(\Omega) \biggm|
	\frac{\partial\rho}{\partial\nu} = 0\ \text{on}\ \partial\Omega \right\}.
\]
It is well known that $B_0$ is a non-negative self-adjoint operator with respect to the standard inner product of $L^2(\Omega)$.
 Moreover, $B_0$ has a bounded inverse $B_0^{-1}$ so that $B_0$ is positive in the sense that
\[
	\int_\Omega(B_0\rho)\rho\, dx \ge \delta_0 \|\rho\|_{L^2}^2 \quad\text{for}\quad
	\rho \in L_\mathrm{av}^2(\Omega)
\]
with some constant $\delta_0>0$, since we use the average free space $L_\mathrm{av}^2(\Omega)$.
 It is not difficult to see that
\[
	D(B_0^{1/2}) = H_\mathrm{av}^1(\Omega), \quad
	H_\mathrm{av}^1(\Omega) = H^1(\Omega) \cap L_\mathrm{av}^2(\Omega).
\]
A general theory of a positive self-adjoint operator gives a relation between complex interpolation spaces and domains of fractional powers
\[
	D(B_0^\alpha) = \left[ D(B_0^{\alpha_1}), D(B_0^{\alpha_2}) \right]_\theta, \quad
	\alpha = \alpha_1(1-\theta) + \alpha_2\theta, \quad
	\theta \in (0,1)
\]
for $\alpha_2>\alpha_1$; see \cite[Theorem~4.36]{L2}.
 By reiteration of interpolation, we see that
\[
	D(B_0^{3/2}) = \left\{ \rho \in D(B_0) \Bigm|
	B_0\rho \in D(B_0^{1/2}) \right\};
\]
see e.g.\ \cite[Theorem~1.23]{L2};
 note $[X_0,X_1]_\theta=(X_0,X_1)_{\theta,2}$ ($0<\theta<1$), where the latter is a real interpolation for Hilbert spaces $X_0,X_1$ (\cite[Corollary~4.37]{L2}).
 Since $D(B_0^{1/2})\subset H^1(\Omega)$, by elliptic regularity theory (see e.g.\ \cite{E}), this space equals $H^3(\Omega) \cap D(B_0)$.

We set a basic space
\[
	X_0 := H_\mathrm{av}^1 (\Omega) \times \left(L^2(\Omega)\right)^n
\]
equipped with inner product
\[
	(U_1,U_2)=\int_\Omega \nabla \rho_1\cdot \nabla\bar{\rho}_2 
	+ \int_\Omega u_1 \bar{u}_2\rho_*\, dx
\]
where $\rho_*\in C^3(\bar{\Omega})$, $\rho_*>0$ is a given function for $U_i=(\rho_i,u_i)^T$ with $i=1,2$.
We set a solution space
\[
	X_1 := \left(H^3(\Omega) \cap D(B_0)\right) \times D(L_0).
\]

We set an intermediate space as a real interpolation space
\[
	X_{1/2} := (X_0, X_1)_{1/2,2}.
\]
As we already mentioned, this space equals the complex interpolation space $[X_0,X_1]_{1/2}$ which can be written as fractional power spaces, i.e.,
\[
    X_{1/2} = \left[ D(B_0^{1/2}), D(B_0^{3/2}) \right]_{1/2}
    \times \left[ L^2(\Omega), D(L_0) \right]_{1/2}
    = D(B_0) \times D (L_0^{1/2}).
\]
Thus its explicit form is
\[
	X_{1/2} = \left\{ \rho \in H^2(\Omega) \biggm|
	\frac{\partial\rho}{\partial\nu} = 0\ \text{on}\ \partial\Omega,
	\int_\Omega \rho\, dx = 0 \right\} 
	\times \left\{ u \in H^1(\Omega) \bigm|
	u = 0 \ \text{on}\ \partial\Omega \right\}.
\]
Let $\rho_*\in C^3(\bar{\Omega})$ be a positive function.
 For $U\in X_{1/2}$, we set operators $A(U):X_1\to X_0$, $F(U)\in X_0$ by
\begin{align*}
	& A(U) := \left(
	\begin{array}{cc} 
   0 & \operatorname{div}\left((\rho_*+\rho) \cdot \right) \\ 
   \nabla(\varepsilon B_0) & \displaystyle\frac{L_0}{\rho_*+\rho} + (u \cdot \nabla)
   \end{array}
	\right), \\
	& F(U) := \dbinom{0}{-\nabla\Psi_\rho(\rho_*+\rho) + \nabla\Psi_\rho(\rho_*)}.	
\end{align*}
If $\rho_*$ is a stationary solution, our system can be rewritten as $U_t+A(U)U=F(U)$.
 The spaces $X_0$, $X_1$, $X_{1/2}$ should be interpreted as spaces of column vectors $(\rho,u)^T$ with suitable regularity.

We next study $U$-dependence of $A$ and $F$.
 Let $\mathcal{L}(X_1,X_0)$ denote the space of bounded linear operators from $X_1$ to $X_0$.
 This space is regarded as a Banach space equipped with the operator norm.
\begin{prop} \label{PBa}
Assume that $1\le n\le3$ and $\Psi\in C^3$.
 Then $A:U\mapsto A(U)$ for $U=(\rho,u)^T\in X_{1/2}$ is a $C^1$ mapping with values in $\mathcal{L}(X_1,X_0)$ provided that $\|\rho\|_{H^2}$ is sufficiently small.
 The mapping $U\mapsto F(U)$ is a $C^1$ mapping from $X_{1/2}$ to $X_0$.
\end{prop}
\begin{proof}
To say that $A$ is $C^1$ from $X_{1/2}$ to $\mathcal{L}(X_1,X_0)$ for sufficiently small $\rho$, it suffices to prove that
\begin{enumerate}
\item[(i)] multiplication operator $1/(\rho_*+\rho)$ is $C^1$ from $X_{1/2}$ to $L^\infty$;
\item[(i\hspace{-0.1em}i)] $u\cdot$ is $C^1$ from $H^1$ to $L^2$;
\item[(i\hspace{-0.1em}i\hspace{-0.1em}i)] $(\rho,u)\mapsto\operatorname{div}\left((\rho_*+\rho)u\right)$ is bounded from $H^2\times H^2$ to $H^1$.
\end{enumerate}
 If $1\le n\le3$, then by the Sobolev embedding $H^2\hookrightarrow L^\infty$, the function $\rho\mapsto1/(\rho_*+\rho)$ is $C^1$ from $H^2$ to $L^\infty$ provided that $\|\rho\|_{H^2}$ is small.
 By the Sobolev embedding, $H^1\hookrightarrow L^4$ provided that $1\le n\le4$ since $\Omega$ is bounded.
 Thus $(u,v)\mapsto u\cdot v$ is a bounded bilinear map from $L^4\times L^4\to L^2$ by the H\"{o}lder inequality.
 Thus $u\cdot$ is $C^1$ from $H^1$ to $L^2$ for $n\le4$.

It remains to prove that $(\rho,u)\mapsto\operatorname{div}\left((\rho_*+\rho)u\right)$ is a bounded bilinear map from $H^2\times H^2$ to $H^1$.
 Fortunately,
\begin{gather*}
	\| \nabla ^2 \rho \cdot u \|_{L^2}
	\le C \|\rho\|_{H^2} \|u\|_\infty
	\le C' \|\rho\|_{H^2} \|u\|_{H^2} \\
	\| \rho \nabla ^2 u \|_{L^2}
	\le C \|\rho\|_\infty \|u\|_{H^2}
	\le C' \|\rho\|_{H^2} \|u\|_{H^2}
\end{gather*}
for $1\le n\le3$ and
\[
	\| \nabla \rho \cdot \nabla u \|_{L^2}
	\le C \|\rho\|_{H^2} \|u\|_{H^2}
\]
for $1\le n\le4$ by Sobolev's embedding with some $C$ and $C'$ independent of $\rho$ and $u$. 
 These bilinear estimates imply that $\rho\mapsto\operatorname{div}(\rho\cdot)$ is $C^1$ from $H^2$ to $\mathcal{L}(X_1,X_0)$.

For $F$, since $\Psi$ is $C^3$ so that $\Psi_{\rho\rho}(\rho_*+\rho)$ is $C^1$, the mapping $\rho\mapsto\Psi_{\rho\rho}(\rho_*+\rho)$ is $C^1$ from $L^\infty$ to $L^\infty$.
 Since
\[
	\nabla \Psi_\rho(\rho_*+\rho)
	= \Psi_{\rho\rho}(\rho_*+\rho) \nabla(\rho_*+\rho)
\]
and $H^2\hookrightarrow L^\infty$ for $1\le n\le3$, we observe that $\rho\mapsto\nabla\Psi_\rho(\rho_*+\rho)$ is $C^1$ from $H^2$ to $H^1$. 
\end{proof}
\subsection{An explicit form of the linearized operator} \label{SSAE} 

We next consider linearized operator of $A(U)U-F(U)$ at $U=(0,0)$.
 We set
\begin{align*}
	A_0 U &= A(0)U + \left(A'(0)U\right) 0 - F'(0)U. \\
	&= A(0)U - F'(0)U.
\end{align*}
We shall calculate $F'(\tilde{U})$ for $\tilde{U}=(\tilde{\rho},\tilde{u})^T$.
 We first notice that
\[
	\partial_\rho F(\tilde{U})\rho
	= \binom{0}{-\nabla \left( \Psi_{\rho\rho}(\rho_*+\tilde{\rho}) \rho \right)}, \quad
	\partial_u F(\tilde{U})u
	= \binom{0}{0}
\]
to get
\[
	F'(\tilde{U})U
	= \dbinom{0}{-\nabla \Psi_{\rho\rho}(\rho_*+\tilde{\rho}) \rho}.
\]
Thus
\[
	F'(0)U
	= \dbinom{0}{-\nabla\Psi_{\rho\rho}(\rho_*)\rho}.
\]
We set an operator in $L_\mathrm{av}^2(\Omega)$
\[
	B(\rho_*) = P_1 \left( \varepsilon B_0 + \Psi_{\rho\rho}(\rho_*) \right),
\]
where $P_1$ is the orthogonal projection of $L^2(\Omega)$ for $L_\mathrm{av}^2(\Omega)$, i.e.,
\[
	P_1 g = g - \frac{1}{|\Omega|} \int_\Omega g\,dx.
\]
We shall always assume $\Psi\in C^3$.
 Using these formulas, we obtain
\[
	A_0 = \left(
	\begin{array}{cc} 
   0 & \operatorname{div}(\rho_* \cdot) \\ 
   \nabla \left(B(\rho_*) \cdot \right) & \displaystyle\frac{L_0}{\rho_*}   \end{array}
	\right)
\] 
so that
\[
	A_0 U = \left(
	\begin{array}{cc}
	\operatorname{div}(\rho_* u) \\
	\nabla \left(B(\rho_*) \rho \right) 	
	+\displaystyle\frac{L_0 u}{\rho_*}
	\end{array}
	\right), \quad
	U = \dbinom{\rho}{u}.
\]

To see the spectral property of the operator $A_0$, we complexify spaces $X_0$ and $X_1$.
 For complex-valued functions, we set
\begin{align*}
	& A(U)U_1 = \left(
	\begin{array}{cc}
   \operatorname{div}(\rho_* u_1) 
   + \operatorname{div}(\rho u_1) \\
   \nabla(\varepsilon B_0 \rho_1)
   + \displaystyle\frac{L_0 u_1}{\rho_*+\rho} + (u\cdot\nabla) u_1
   \end{array}
	\right), \\
	& F(U) = \dbinom{0}{-\nabla\Psi_\rho(\rho_*+\rho) + \nabla\Psi_\rho(\rho_*)}.
\end{align*}
Here we extend $\Psi$ defined on $\mathbb{R}$ to $\mathbb{C}$ so that $\Psi$ is $C^3$ as a function of $(x,y)\in\mathbb{R}^2$ for $z=x+iy\in\mathbb{C}$.
 The linearized operator $A_0$ for a complex-valued function is the same as before, i.e.,
\[
	A_0 U = \left(
	\begin{array}{cc}
	\operatorname{div}(\rho_* u) \\
	\nabla \left(B(\rho_*) \rho \right) 	
	+\displaystyle\frac{L_0 u}{\rho_*}
	\end{array}
	\right), \quad
	U = \dbinom{\rho}{u}.
\]

\subsection{Resolvent problem} \label{SSR}

We consider the resolvent problem
\begin{equation} \label{ERP}
	(\xi + A_0) U = G \quad\text{with}\quad
	G = \dbinom{r}{f},
\end{equation}
where $\xi$ is a complex number, i.e., $\xi\in\mathbb{C}$.
 An explicit form of this equation is
\begin{empheq}[left={\empheqlbrace}]{alignat=2}
  & \xi\rho + \operatorname{div}(\rho_* u)= r, \label{ERD} \\
  & \xi u + \frac{L_0 u}{\rho_*} + \nabla \left( B(\rho_*)\rho \right)= f. \label{ERV}
\end{empheq}
The next result is valid for all dimensions.
\begin{lem} \label{LKI}
Assume that $U=(\rho,u)^T\in X_1$ satisfies \eqref{ERD} and \eqref{ERV} with $\xi\in\mathbb{C}$.
 Then
\begin{equation} \label{EKI}
	\xi \int_\Omega B(\rho_*) \bar{\rho} \rho
	\, dx 
	+ \bar{\xi} \int_\Omega \rho_* |u|^2\,dx
    + \int_\Omega u \cdot L_0 \bar{u}\, dx
	= \int_\Omega \left( rB(\rho_*) \right)\bar{\rho} + \rho_* \bar{f}\cdot u\, dx.
\end{equation}
\end{lem}
\begin{proof}
We take complex conjugate of \eqref{ERV} to get
\begin{equation} \label{ERCJ}
	\bar{\xi} \bar{u} + \frac{L_0\bar{u}}{\rho_*}
    + \nabla \left( B(\rho_*)\bar{\rho} \right) = \bar{f}.
\end{equation}
We consider
\[
	\eqref{ERD}\cdot B(\rho_*)\bar{\rho}
	+ \eqref{ERCJ}\cdot \rho_* u
\]
and integrate over $\Omega$.
 Integration by parts yields the desired identity if we notice that
\[
	\int_\Omega \left\{ \left( B(\rho_*) \bar{\rho} \right) \operatorname{div}(\rho_*u)
	+ \rho_* u \cdot\nabla \left(B(\rho_*) \bar{\rho} \right) \right\} dx=0.
\]
\end{proof}

From now on we consider the case $\rho_*=\rho^e$ i.e.\ $(\rho^e,0)$ is an equilibrium of \eqref{ECM}, \eqref{ECMM2} with \eqref{EB}.
\begin{lem} \label{LSt}
If $\rho_*=\rho^e$ is energetically stable, then
\[
	\int_\Omega \left( B(\rho^e) \bar{\rho} \right)\rho\, dx \ge 0, \quad
	\rho \in D (B_0).
\]
\end{lem}
\begin{proof}
Integration by parts yields
\[
	\int_\Omega \left( B(\rho^e) \bar{\rho} \right) \rho\, dx
	= \int_\Omega \left( \varepsilon |\nabla\rho|^2 + \Psi_{\rho\rho}(\rho^e) |\rho|^2 \right) dx \ge 0
\]
since $\rho^e$ is energetically stable.
\end{proof}
We are interested to find point spectrum.
 We consider the resolvent problem \eqref{ERP} with $G=0$, i.e.,
\begin{equation} \label{ER0}
	(\xi + A_0) U = 0, \quad
	U \in D(A_0)
\end{equation}
when $\rho_*=\rho^e$ for energetically stable $\rho^e$.
\begin{thm} \label{TSp}
Assume that $\mu>0$ and $2\mu+\lambda>0$ and that $\rho^e>0$ ($\rho^e\in C^3(\bar{\Omega})$) is energetically stable.
 If there exists a solution $U\not\equiv0$ to \eqref{ER0}, then the real part $\operatorname{Re}\xi\le0$.
 If $\operatorname{Re}\xi=0$, then $\xi=0$.
\end{thm}
\begin{proof}
If $U$ solves \eqref{ER0}, Lemma \ref{LKI} yields
\begin{equation} \label{EKI0}
	\xi \int_\Omega  \left( B(\rho^e) \bar{\rho} \right) \rho \,dx
    + \bar{\xi} \int_\Omega \rho^e |u|^2\, dx 
    + \int_\Omega u \cdot L_0 \bar{u}, dx=0.
\end{equation}
By Lemma \ref{LLame}, we see that
\[
    \int_\Omega u\cdot L_0\bar{u}\,dx
    = \overline{\int_\Omega \bar{u}\cdot L_0 u\,dx} \ge0.
\]
 By Lemma \ref{LSt}, we see $\int \left(B(\rho^e)\bar{\rho}\right)\rho\,dx\ge0$.
 Thus if $\operatorname{Re}\xi>0$, then $u\equiv0$ by \eqref{EKI0}.
 By \eqref{ERD} with $r=0$, this implies that $\rho\equiv0$.
 Thus $\operatorname{Re}\xi\le0$.

If $\operatorname{Re}\xi=0$, then \eqref{EKI0} implies that $\int_\Omega\bar{u}\cdot L_0u\,dx=0$.
 By Lemma \ref{LLame}, the last identity implies $\nabla u\equiv0$.
 If $\partial\Omega\neq\emptyset$, then by the zero boundary condition, we see that $u\equiv0$.
 By \eqref{ERD}, we observe that $\rho\equiv0$.
 We thus conclude that the imaginary part $\operatorname{Im}\xi=0$, i.e., $\xi=0$.
 The proof of the case $\Omega=\mathbb{T}^n$ is postponed after the proof of next lemma.
\end{proof}
\begin{lem} \label{LSp0}
Assume that $\mu>0$ and $2\mu+\lambda>0$ and that $\rho^e\in\mathcal{E}$ so that $(\rho^e,0)$ is an equilibrium of \eqref{ECM}, \eqref{ECMM2} with \eqref{EB}.
\begin{enumerate}
\item [(i)] If $\partial\Omega\neq\emptyset$, then
\begin{equation} \label{Eker}    
	\ker A_0
	= \left\{(\rho,0) \bigm|
	\rho \in \ker B(\rho^e)\right\}.
\end{equation}
In particular if $\rho^e$ is non-degenerate, then
\[
    m=\dim\ker B(\rho^e),
\]
where $m$ is the dimension of $\mathcal{E}$ near $\rho^e$ as a manifold and $\dim$ denotes the dimension of a vector space.
 If $\rho^e$ is isolated, then, $\ker B(\rho^e)=\{0\}$.
\item [(i\hspace{-0.1em}i)] If $\Omega=\mathbb{T}^n$, then
\begin{equation} \label{EkerP}    
	\ker A_0
	= \left\{(\rho,c) \bigm|
	\rho \in \ker B(\rho^e),\ 
    c\cdot\nabla\rho^e \equiv0, \ 
    c\in\mathbb{R}^n
    \right\}.
\end{equation}
In particular, if $\rho^e$ has no constant direction, then the identity \eqref{Eker} holds.
 If $\rho^e$ is non-degenerate, then
\[
    m=\dim\ker B(\rho^e)+k,
\]
where
\[
    k=\dim \left\{ c\in\mathbb{R}^n \mid
    c\cdot\nabla\rho^e\equiv0\ \text{in}\ \mathbb{T}^n \right\}.
\]
Moreover, $m\ge n-k$.
 In particular, $m\ge1$ unless $\rho^e$ is a constant.
\end{enumerate}
\end{lem}
\begin{proof}
\begin{enumerate}
\item [(i)] By \eqref{EKI0} with $\xi=0$, Lemma \ref{LLame} implies that $\nabla u\equiv0$ so that $u\equiv0$ by the Dirichlet boundary condition for $u$.
 Thus, $A_0U=0$ yields
\[
	u = 0 \quad\text{and}\quad
	\nabla \left( B(\rho^e)\rho \right) = 0
	\quad\text{in}\quad \Omega.
\]
By the definition of $B(\rho^e)$, $\nabla\left(B(\rho^e)\rho\right)=0$ is equivalent to
\[
    \rho\in\ker B(\rho^e).
\]
We thus conclude \eqref{Eker}.
 The remaining assertions are clear by definition and \eqref{Eker}.
\item [(i\hspace{-0.1em}i)] In the case of $\Omega=\mathbb{T}^n$, if we use \eqref{EKI0} with $\xi=0$, Lemma~\ref{LLame} implies that $u=c$, where $c$ is a constant vector.
 Thus $A_0U=0$ yields
\[
    u=c, \quad
    \nabla \left(B(\rho^e)\rho\right)=0
    \quad\text{and}\quad
    \operatorname{div}(\rho^e c)=0.
\]
Since $\operatorname{div}(\rho^e c)=c\cdot\nabla\rho^e$, we now obtain \eqref{EkerP}.
The remaining assertions follow from definition and \eqref{EkerP} except $m\ge n-k$.

Let us give a proof for $m\ge n-k$.
 Since $\rho^e\in\mathcal{E}$, a translated function
\[
    \rho_c^e(x):=\rho^e(x+c)
\]
satisfies
\[
    \Psi_\rho(\rho_c^e)- \varepsilon\Delta\rho_c^e=a
\]
with a constant $a$ independent of $c$.
 We set $c=c_0s$ for $s\in\mathbb{R}$ and differentiate the above equation with respect to $s$ to get
\[
    c_0\cdot\nabla \left(\Psi_\rho(\rho^e)-\varepsilon\Delta\rho^e \right)=0
\]
by setting $s=0$.
 This implies
\[
    \left(\varepsilon B_0+\Psi_{\rho\rho}(\rho^e)\right)
    (c_0\cdot\nabla\rho^e)=0,
\]
so that
\[
    c_0\cdot\nabla\rho^e \in \ker B(\rho^e).
\]
Thus,
\[
    \dim\ker B(\rho^e)
    \ge n-k.
\]
By definition of non-degeneracy, $m=\dim\ker B(\rho^e)$ so we conclude that $m\ge n-k$.
 If $n=k$, then $\rho^e$ must be a constant so $n-k\ge1$ if $\rho^e$ is not a constant.
\end{enumerate}
\end{proof}
\noindent
\begin{proof}[Proof of Theorem~\ref{TSp} in the case of $\Omega=\mathbb{T}^n$]
 If $\operatorname{Re}\xi=0$, as in the proof for the case $\partial\Omega\neq\emptyset$, we conclude that $u$ is a constant vector $c$.
 From the proof of Lemma~\ref{LSp0} (i\hspace{-0.1em}i), we know that $c\cdot\nabla\rho^e\in\ker B(\rho^e)$.
 By \eqref{ERD}, $-\rho=c\cdot\nabla\rho^e/\xi$.
 If $\operatorname{Im}\xi\neq0$ and $\operatorname{Re}\xi=0$, then $B(\rho^e)\rho=0$.
 The identity \eqref{EKI0} now implies that $\bar{\xi}\int_\Omega\rho^e|u|^2\,dx=0$, which implies $u=0$.
 By \eqref{ERD}, we see that $\rho=0$.
 We now conclude that $(\xi_0+A_0)U=0$ for $\operatorname{Re}\xi=0$, $\operatorname{Im}\xi\neq0$ implies $U\equiv0$.
 The proof for periodic case is now complete.
\end{proof}

\subsection{Maximum regularity} \label{SSS}

Instead of considering $A_0$, we consider the operator $A_0$ with no $\Psi$ and write it by $A_{00}$.
 By definition, $A_{00}=A(0)$ and its explicit form is
\[
	A_{00} U =
	\dbinom{\operatorname{div}(\rho^e u)}
	{\nabla (\varepsilon B_0 \rho)}
	+ \dbinom{0}{L_0 u/\rho^e}
	\quad\text{for}\quad
	U = \dbinom{\rho}{u}.
\]

The maximum regularity result we need is a special case of \cite[Theorem~2.1]{Ko} when $\Omega$ is a bounded domain.
\begin{lem} \label{LMax}
Assume that $\mu>0$ and $2\mu+\lambda>0$.
 Assume that $0<T<\infty$.
 A linear equation
\[
    \frac{dU}{dt}+A_{00}U=G
    \quad\text{in}\quad (0,T)
    \quad\text{with}\quad U(0)=U_0\in X_{1/2}
\]
is uniquely solvable for $G\in L^2(0,T;X_0)$ and
\[
    \int_0^T \left(\left\|\frac{dU}{dt}\right\|_{X_0}^2 + \|U\|_{X_1}^2\right)\,dt
    \le C_T \left(\int_0^T \|G\|_{X_0}^2\,dt
     + \|U_0\|_{X_{1/2}}^2\right)
\]
with some constant $C_T$ independent of $G$ and $U_0$.
\end{lem}

In the case of bounded domain, \cite[Theorem~2.1]{Ko} gives a more general estimate replacing $2$ by $p$ and $X_0$ by its $L^p$ version for all $p\in(1,\infty)$ except $p=3/2$.
 In his setting, inhomogeneous boundary data is allowed and $\lambda$, $\mu$, $\varepsilon$ are allowed to depend on $x$ and $t$.
 The main idea in his proof is introducing new dependent variables
\[
    w=\operatorname{div}u,\quad
    \tau=\Delta\rho
\]
and diagonalized the operator by considering a suitable linear combination of $w$ and $\tau$.
 Note that in his case $\operatorname{div}(\rho^eu)$ is replaced by $\rho^e\operatorname{div}u$ but this is a lower order perturbation so it does not affect the maximum regularity estimate for $T<\infty$.
 The case $\Omega=\mathbb{T}^n$ is much simpler.
 The maximum regularity implies the sectoriality with angle less than $\pi/2$ \cite[Proposition~3.5.2]{PS}.
\begin{cor} \label{CMax}
Under the same assumptions of Lemma~\ref{LMax} there exist positive constants $\delta$ and $R$ such that $(\zeta+A_{00})^{-1}\in\mathcal{L}(X_0,X_1)$ with
\[
    |\zeta|\left\|(\zeta+A_{00})^{-1}U\right\|_{X_0}
    \le C\|U\|_{X_0}
\]
for $\zeta\in\Sigma_{\pi/2+\delta}\cap(B_R)^C$ with $C=C_{\delta,R}$ independent of $\zeta$ and $U$.
 Here $\Sigma_\sigma=\left\{z\in\mathbb{C}\setminus\{0\} \bigm||\arg z|<\sigma\right\}$ and $B_R=\left\{z\in\mathbb{C}\bigm||z|<R\right\}$.
\end{cor}

Corollary~\ref{CMax} yields a result for $A_0$ since $A_{00}-A_0$ can be regarded as a lower order perturbation provided that $\Psi\in C^3$ and $\rho^e\in C^3(\bar{\Omega})$.
\begin{thm} \label{TRes}
Let $\Omega$ be a bounded domain or $\Omega=\mathbb{T}^n$.
Assume that $\mu>0$ and $\mu+2\lambda>0$.
Assume that $\Psi\in C^3$ and $\rho^e\in C^3(\bar{\Omega})$ with $\rho^e>0$.
 Then there exist positive constants $\delta$ and $R$ such that $(\zeta+A_0)^{-1}\in\mathcal{L}(X_0,X_1)$ and satisfies
\[
    |\zeta| \left\|(\zeta+A_0)^{-1}U\right\|_{X_0}
    \le C \|U\|_{X_0}
\]
for all $\zeta\in\Sigma_{\pi/2+\delta}\cap(B_R)^C$ with $C$ independent of $\zeta$ and $U$.
\end{thm}
\begin{rem} \label{RDK}
In \cite[Theorem~2.2]{Ko}, the maximal regularity in $\mathbb{R}^n$ is proved for all $p\in(1.\infty)$.
 According to general theory of a parabolic mixed order system based on Newton polygons, the operator $A_{00}$ with constant coefficients has the $L^q
 $ in time $-L^q$ in space maximum regularity result for any $p,q\in(1,\infty)$ as explained in \cite[Theorem~4.22]{DK}.
\end{rem}

\subsection{Spectral properties} \label{SSS}

Since $X_1\hookrightarrow X_0$ is compact, $(\zeta+A_0)^{-1}$ is a compact operator from $X_0$ into itself if it belongs to $\mathcal{L}(X_0,X_1)$.
 By the Riesz-Schauder theory, Theorem~\ref{TRes} implies that all spectrum $\sigma(A_0)$ of $A_0$ (and also $A_{00}$) consists of eigenvalues and its accumulation points lie only at the space infinity.
 Theorem \ref{TSp} says that all eigenvalues of $A_0$ except $0$ have positive real part if $\rho^e$ is energetically stable.
 For $A_{00}$ it is not difficult to prove that $\ker A_{00}=\{0\}$, i.e., $0$ is not an eigenvalue of $A_{00}$, when $\Omega$ is a bounded domain; see Theorem~\ref{TSol}.
 In the case of $\Omega=\mathbb{T}^n$, $\ker A_{00}=\mathbb{C}^n$.
\begin{thm} \label{TSp2}
Let $\Omega$ be a smooth bounded domain or $\Omega=\mathbb{T}^n$.
 Assume that $\Psi\in C^3$ and $\rho^e\in C^3(\bar{\Omega})$ is energetically stable and positive. 
 Then
\begin{enumerate}
\item[(i)] $\sigma(A_0)\setminus\{0\}$ consists of eigenvalues $\{\zeta_j\}_{j=1}^\infty\subset\Sigma_{\pi/2-\delta}$ with no accumulation points in $\Sigma_{\pi/2-\delta}$ where $\delta>0$ is in Theorem~\ref{TRes}; 
\item[(i\hspace{-0.1em}i)] the same assertion holds for $A_{00}$; 
\item[(i\hspace{-0.1em}i\hspace{-0.1em}i)] if $\rho^e$ is non-degenerate and $\Omega$ is a bounded domain, then
\[
    \dim\ker A_0
    = \dim\ker B
    (\rho^e);
\]
\item[(i\hspace{-0.1em}v)] if $\rho^e$ is non-degenerate and there is no constant direction in the case of $\Omega=\mathbb{T}^n$, then
\[
    \dim \ker A_0
    = \dim \ker B(\rho^e).
\]
In general, in the case of $\Omega=\mathbb{T}^n$,
\[
    \dim\ker A_0=k+\dim\ker B(\rho^e),
\]
where $k=\dim\{c\in\mathbb{C}^n\mid c\cdot\nabla\rho^e\equiv0\}$.
\end{enumerate}
\end{thm}
\begin{proof}
As noticed before, $\sigma(A_0)$ and $\sigma(A_{00})$ consists of eigenvalues.
 By Theorem~\ref{TSp}, all eigenvalues except zero have positive real part.
 The properties (i) and (i\hspace{-0.1em}i) follow from Theorem~\ref{TRes} and Corollary~\ref{CMax}.
 The properties (i\hspace{-0.1em}i\hspace{-0.1em}i) and (i\hspace{-0.1em}v) follow from Lemma~\ref{LSp0}.
\end{proof}

\subsection{One-dimensional case} \label{SSOne}

We notice that in one-dimensional setting, the sectoriality (Corollary~\ref{CMax}) is easy to prove by an explicit formula.
 In thie subsection, we give an explicit solution formula.
 Note that the sectorialily with angle less than $\pi/2$ implies the maximum regularity when $X_0$ and $X_1$ are Hilbert spaces \cite[Proposition~2.3]{Sa}.
 We consider the resolvent problem
\[
	(\zeta + A_{00})U = G, \quad
	G = (r,f)^T
\]
and give an explicit proof of Corollary~\ref{CMax} when $\Omega$ is a bounded interval.
 The resolvent problem can be written as 
\begin{align}
	&\zeta\rho + \partial_x(\rho_* u) = r
	\quad\text{in}\quad \Omega=(\alpha,\beta), \label{E1R1} \\
	&\zeta u - \frac{\tilde{\mu}}{\rho_*} u_{xx} - \varepsilon \rho_{xxx} = f
	\quad\text{in}\quad \Omega, \label{E1R2} \\
	&u = 0, \quad \rho_x = 0
	\quad\text{on}\quad \partial\Omega=\{\alpha,\beta\}, \label{E1R3}
\end{align}
where $\tilde{\mu}=2\mu+\lambda$.
 We differentiate \eqref{E1R1} and set $w=\rho_x$ to get
\begin{equation} \label {E1RR1}
	\zeta w + \rho_* u_{xx} + 2\rho_{*x} u_x + \rho_{*xx}u = r_x.
\end{equation}
The equation \eqref{E1R2} can be written as
\begin{equation} \label {E1RR2}
	\zeta u - a u_{xx} - \varepsilon w_{xx} = f
\end{equation}
and \eqref{E1R3} is now
\begin{equation} \label {E1RR3}
	u = w = 0
	\quad\text{on}\quad \partial\Omega,
\end{equation}
where $a=\tilde{\mu}/\rho_*>0$.
 The equations \eqref{E1RR1}, \eqref{E1RR2} can be written as
\[
	\zeta \dbinom{u}{w}
	- \left(
	\begin{array}{cc} 
   a & \varepsilon \\ 
   -\rho_* & 0
	\end{array}
	\right) \dbinom{u}{w}_{xx}
	+ \dbinom{0}{2\rho_{*x}u_x+\rho_{*xx}u}
	= \dbinom{f}{r_x}.
\]
Let us consider the eigenvalues of the matrix
\[
	K = \left(
	\begin{array}{cc} 
   a & \varepsilon \\ 
   -\rho_* & 0
	\end{array}
	\right).
\]
The eigen-polynomial is $z^2-az+\varepsilon\rho_*$ so the eigenvalues are
\[
	\lambda_\pm = \left(a \pm \sqrt{a^2 - 4\varepsilon\rho_*}\right) \Bigm/2.
\]
Its real part is estimated as
\[
	2a \ge \operatorname{Re} \lambda_+ \ge \operatorname{Re} \lambda_-
	= \frac{2\varepsilon\rho_*}{a + \sqrt{a^2 - 4\varepsilon\rho_*}}
	\ge 2 \varepsilon \inf_\Omega \frac{\rho_*}{2a} > 0
\]
if $a^2-4\varepsilon\rho_*\ge0$ and otherwise $\operatorname{Re}\lambda_\pm=a/2$.
 In any case, $\operatorname{Re}\lambda_\pm$ is bounded away from zero and from above.
 A standard elliptic theory \cite{E} implies that for any $(f,r_x)^T\in\left(L^2(\Omega)\right)^2$, there exists unique $(u,w)^T\in\left(H^2(\Omega)\right)^2$ provided that $\operatorname{Re}\zeta>0$ is sufficiently large.
 Moreover,
\[
	\left\|\dbinom{u}{w}\right\|_{H^2}
	\le C \left\|\dbinom{f}{r_x}\right\|_{L^2}.
\]
Since $\rho$ is recovered from $w$ by
\[
	\rho(x) = \int_\alpha^x w
	-\frac{1}{\beta-\alpha}\int_\alpha^\beta \left(\int_\alpha^x w\,dy \right)dx
\]
to get $\int_\alpha^\beta\rho\,dx=0$, we conclude that $(\zeta+A_{00})^{-1}\in\mathcal{L}(X_0,X_1)$ for sufficiently large $\zeta$.
 One is able to get resolvent estimates of Corollary~\ref{CMax} from the explicit solution formula.
 Moreover, we see that $0\notin\sigma(A_{00})$.

For periodic case, similar argument yields the desired result.

\subsection{Higher dimensional case} \label{SSH}

We consider the resolvent problem
\[
	(\zeta + A_{00}) U = G, \quad
	G = (r,f)^T
\]
and give a proof that $0\not\in\sigma(A_{00})$.
 We only discuss the case when $\Omega$ is a bounded domain with necessary boundary condition.
 The case $\Omega=\mathbb{T}^n$ is more direct (see Remark~\ref{RPe}).
 We consider the case $\zeta=0$.
 Namely we consider $A_{00}U=G$ or
\begin{align}
	&\operatorname{div}(\rho_* u) = r
	\quad\text{in}\quad \Omega \label{EH1} \\
	&\frac1{\rho_*} L_0 u + \nabla(-\varepsilon\Delta\rho) = f	
	\quad\text{in}\quad \Omega \label{EH2} \\
	&u = 0, \quad \partial\rho/\partial\nu = 0
	\quad\text{on}\quad \partial\Omega. \label{EH3}
\end{align}
We shall prove that $A_{00}$ has a bounded inverse.
\begin{thm} \label{TSol}
Assume that $\mu>0$ and $\lambda+2\mu>0$ and that $\rho_*\in C^3(\bar{\Omega})$ is positive.
 Then $A_{00}$ has a bounded inverse from $X_0$ to $X_1$.
 Moreover
\begin{equation} \label{EKA}
	\|A_{00}^{-1} G\|_{X^1} \le C\|G\|_{X_0}, \quad
	G \in X_0
\end{equation}
with $C$ depending on $\rho_*$ only through its $C^3$ norm and $\inf_\Omega \rho_*$.
\end{thm}
The key idea is to apply a classical solvability result for the stationary Stokes problem with non-zero divergence.
 We consider
\begin{align}
	\operatorname{div}u 
	&= r \quad\text{in}\quad \Omega \label{ESS1} \\
	-\Delta u + \nabla\varpi
	&= f \quad\text{in}\quad \Omega \label{ESS2} \\
	u &= 0 \quad\text{on}\quad \partial\Omega. \label{ESS3}
\end{align}
The next result is found for example in a book of H.~Sohr \cite[I\!I\!I. Theorem~1.5.3]{Soh}.
\begin{lem} \label{LSS}
Let $\Omega$ be a smooth ($C^2$) bounded domain in $\mathbb{R}^n$ ($n\ge2$).
 Assume that $r\in H_\mathrm{av}^1(\Omega)$ and $f\in \left(L^2(\Omega)\right)^n$.
 Then, there exists a unique solution $(u,\varpi)\in\left(H^2(\Omega)\right)^n\times L_\mathrm{av}^2(\Omega)$ of \eqref{ESS1}--\eqref{ESS3} satisfying
\begin{equation} \label{EGS}
    \|u\|_{H^2(\Omega)} + \| \nabla\varpi \|_{L^2(\Omega)}
	\le C \left( \|f\|_{L^2(\Omega)} + \|r\|_{H^1(\Omega)} \right)
\end{equation}
with some $C$ independent of $f$ and $r$.
\end{lem}
To give the main idea to solve \eqref{EH1}--\eqref{EH3}, we first consider the case $\rho_*\equiv1$.
 The equation becomes
\begin{align}
	&\operatorname{div}u = r \quad\text{in}\quad \Omega \label{EHS1} \\
	&-\Delta u + \nabla\varpi = f/\mu \quad\text{in}\quad \Omega \label{EHS2} \\
	&\varpi = -\left( \varepsilon \Delta \rho + ( \lambda+\mu ) \operatorname{div}u \right)
	/\mu \quad\text{in}\quad \Omega \label{EHS3} \\
	&u = 0, \quad \partial\rho/\partial\nu = 0 \quad\text{on}\quad \partial\Omega. \label{EHS4}
\end{align}
For a given $r\in H_\mathrm{av}^1(\Omega)$ and $f\in \left(L^2(\Omega)\right)^n$, we apply Lemma \ref{LSS} to find solution $(u,\varpi)\in(H^2)^n\times H^1$ of \eqref{EHS1}, \eqref{EHS2} with $u=0$ on $\partial\Omega$.
 The relation \eqref{EHS3} implies that
\[
	-\varepsilon \Delta \rho = h
\]
with
\[
	\| h \|_{H^1} \le C \left( \|u\|_{H^2} + \| \varpi \|_{H^1} \right).
\]
Since $\int_\Omega r\,dx=0$, so is $h$.
 This Poisson equation with $\partial\rho/\partial\nu=0$ of \eqref{EHS4} is uniquely solvable and by elliptic regularity theory, we obtain $\|\rho\|_{H^3}\le C\|h\|_{H^1}$; see \cite{E}.
 We thus found a desired $(\rho,u)^T\in X_1$.
 By \eqref{EGS},
\[
	\| u \|_{H^2} + \| \varpi \|_{H^1} 
	\le C \left( \| f \|_{L^2} + \| r \|_{H^1} \right),
\]
so $\|\rho\|_{H^3}$ is also estimated by $C\left(\|f\|_{L^2}+\|r\|_{H^1}\right)$.
 Note that this argument works for $L^q(\Omega)$ space with trivial modification.
 We also note that the same argument works for the Dirichlet condition $\rho=0$ on $\partial\Omega$ by considering $H^1$ space instead of $H_\mathrm{av}^1$ space.
 In the case $\rho_*\equiv1$ (or equivalently $\rho_*$ is a non-zero constant), to get \eqref{EKA} and Theorem \ref{TSol}, we only need that $\mu\neq0$.
 No assumption on $\lambda$ is necessary. 
\begin{proof}[Proof of Theorem \ref{TSol}]
We use a homotopy argument.
 For $s\in(0,1)$ we set
\[
	\rho_s = (1-s) + s\rho_*.
\]
The corresponding $A_{00}$ (with $\rho_*=\rho_s$) is denoted by $A_{00s}$.
 By the previous argument, $A_{000}$ has a bounded inverse $A_{000}^{-1}$ from $X_0$ to $X_1$.
 Since $\rho_*\in C^2(\bar{\Omega})$, by a lower order perturbation  $A_{00s}^{-1}$ is still bounded from $X_0$ to $X_1$ for small $s>0$.

To prove Theorem \ref{TSol} it now suffices to prove an apriori estimate for solutions of \eqref{EH1}, \eqref{EH2}, \eqref{EH3}.
Namely,
\[
	\| u \|_{H^2} + \| \rho \|_{H^3}
	\le C \left( \| r \|_{H^1} + \| f \|_{L^2} \right)  
\]
with $C$ depends on $\rho_*$ only through its $C^3$ norm.
 Suppose that it were false.
 Then, there would exist a sequence $\left\{(\rho_m,u_m)^T\right\}\subset X_1$ and $\left\{(r_m,f_m)^T\right\}\subset X_0$, $\{\rho_{*m}\}\subset C^3$ with
\[
	\inf_m \inf_\Omega \rho_{*m} > 0
\]
such that it solves \eqref{EH1}, \eqref{EH2}, \eqref{EH3} with $\rho_*=\rho_{*m}$ and that 
\[
	\| u_m \|_{H^2} + \| \rho_m \|_{H^3}
	> m \left( \| f_m \|_{L^2} + \| r_m \|_{H^1} \right).  
\]
Since the problem is linear, by dividing $u_m$ and $\rho_m$ by $\|u_m\|_{H^2}+\|\rho_m\|_{H^3}$, we may assume that
\[
	\| u_m \|_{H^2} + \| \rho_m \|_{H^3} = 1, \quad
	\| f_m \|_{L^2} + \| r_m \|_{H^1} < 1/m.  
\]
Since $C^3$ norm of $\rho_{*m}$ is bounded, by Arzel\`{a}-Ascoli's compactness theorem, we may assume that $\rho_{*m}\to\rho_{**}$ in $C^2(\bar{\Omega})$ by taking a subsequence.
 By Rellich's compactness, $\{u_m\}$ and $\{\nabla\rho_m\}$ have a strong convergent subsequence in $H^1$.
 Moreover $\nabla u_m$, $\nabla^2\rho_m$ converges weakly in $L^2$ by Banach-Alaoglu's theorem.
 We still denote $\{u_m\}$ and $\{\rho_m\}$ for such a subsequence.
 Let $u_\infty$ and $\rho_\infty$ be the limit of $\{u_m\}$ and $\{\rho_m\}$.
 We notice that $u_m$ and $\rho_m$ solve
\[
	\frac1{\rho_{*m}} L_0 u_m + \nabla(-\varepsilon \Delta \rho_m) = f_m.
\]
Sending $m\to\infty$ yields
\[
	\frac1{\rho_{**}} L_0 u_\infty + \nabla(-\varepsilon \Delta \rho_\infty) = 0
\]
since $1/\rho_{*m}\to1/\rho_{**}$ in $C^2(\bar{\Omega})$.
 It is much easier to obtain
\[
	\operatorname{div}(\rho_{**}u_\infty) =  0.
\]
By \eqref{EKI}, $\int_\Omega \bar{u}_\infty\cdot L_0 u_\infty\,dx=0$.
 If $\lambda+2\mu>0$, then by Lemma \ref{LLame} we see that $u_\infty\equiv0$.
 This implies $\Delta\rho_\infty$ is a constant.
 However, by the boundary condition $\partial\rho/\partial\nu=0$, this constant $\Delta\rho_\infty$ must be zero.
 Thus, again by the homogeneous Neumann condition $\rho_\infty$ itself is a constant.
 Since $\rho_\infty\in H_\mathrm{av}^1$, this implies that $\rho_\infty=0$.

Since
\[
	L_0 u_m + \rho_{*m} \nabla(-\varepsilon\Delta\rho_m)
	= L_0 u_m + \nabla\varpi_m - (\nabla\rho_{*m})( -\varepsilon\Delta\rho_m)
	\to 0 \quad\text{in}\quad L^2
\]
and $u_m\to0$ in $H^1$, $\rho_m\to0$ in $H^2$, $\rho_{*m}\to\rho_{**}$ in $C^1$, we conclude that
\[
	L_0 u_m + \nabla \varpi_m \to 0 \quad\text{in}\quad L^2
\]
with $\varpi_m=\rho_{*m}(-\varepsilon\Delta\rho_m)$.
 Since
\[
	\operatorname{div}(\rho_{*m}u_m) 
	= \nabla\rho_{*m} \cdot u_m + \rho_{*m}\operatorname{div}u_m \to 0 \quad\text{in}\quad H^1,
\]
and $\rho_{*m}\to\rho_{**}$ in $C^2$ and $u_m\to0$ in $H^1$, $\nabla^2 u_m \rightharpoonup 0$ in $L^2$, 
we conclude that
\[
	\operatorname{div} u_m \to 0
	\quad\text{in}\quad H^1.
\]
Applying Lemma \ref{LSS}, especially \eqref{EGS}, we conclude that
\[
	\| u_m \|_{H^2} + \| \rho_m \|_{H^3} \to 0.
\]
This would contradict our assumption $\|u_m\|_{H^2}+\|\rho_m\|_{H^3} = 1$.
 We thus conclude \eqref{EKA}.
 The proof is now complete.
\end{proof}
\begin{rem} \label{RFS}
We note that the solvability result holds in $L^q$ space \cite{C61}.
 We note that the Stokes resolvent estimate \cite{G81} in $L^q(\Omega)$ is also extended to the case of non-divergence-free case by R.~Farwig and H.~Sohr \cite{FS}, where $\Omega$ is not necessarily a bounded domain.
 It applies to some class of unbounded domains but for such a case the unique solvability for $\zeta=0$ may fail.
 We recall one of their main results on a generalized resolvent equation for the Stokes equations:
\begin{align}
	\operatorname{div}u 
	&= r \quad\text{in}\quad \Omega \label{ES1} \\
	\zeta u - \Delta u + \nabla\varpi
	&= f \quad\text{in}\quad \Omega \label{ES2} \\
	u &= 0 \quad\text{on}\quad \partial\Omega. \label{ES3}
\end{align}
Let us state their main result \cite[Theorem~1.2]{FS} in the case when $\Omega$ is a bounded domain.
\end{rem}
\begin{lem}[\cite{FS}] \label{LFS}
Let $\Omega$ be a bounded $C^2$ domain in $\mathbb{R}^n$ ($n\ge2$).
 Assume that $q\in(1,\infty)$.
 Assume that $r\in W^{1,q}(\Omega)$ and $f\in\left(L^q(\Omega)\right)^n$ with $\int_\Omega r\,dx=0$.
 Assume that $\zeta\in\Sigma_\theta$ with $\theta\in(\pi/2,\pi)$.
 Then there exists a unique solution $(u,\varpi)\in\left(W^{2,q}(\Omega)\right)^n\times L_\mathrm{av}^q(\Omega)$ of \eqref{ES1}, \eqref{ES2} and \eqref{ES3} satisfying
\[
	|\zeta| \|u\|_{L^q} + \| \nabla^2 u \|_{L^q} + \| \nabla\varpi \|_{L^q}
	\le C_\theta \left( \|f\|_{L^q} + \|r\|_{W^{1,q}} + \|\zeta r\|_{W^{-1,q}} \right)
\]
with $C_\theta$ independent of $f$ and $r$.
 The unique solvability is still valid for $\zeta=0$ and $(u,\nabla\varpi)$ satisfies
\[
	\|u\|_{W^{2,q}} + \| \varpi \|_{W^{1,q}}
	\le C \left( \|f\|_{L^q} + \|r\|_{W^{1,q}} \right)
\]
with $C$ independent of $f$ and $r$.
\end{lem}
The idea to use the solvability result with non-zero divergence was first applied for compressible flow equation by Y.~Enomoto and Y.~Shibata \cite{ES} to derive $L^p$ maximal regularity result for a linearized compressible Navier-Stokes equations.
 In the meanwhile the resolvent estimate for a linearized compressible Navier-Stokes equations (including temperature) in $L^p$ is obtained by G.~Str\"{o}hmer \cite{Str} by use of general elliptic $L^p$ theory together with a parameter trick by S.~Agmon.
 Recently, the $L^p$ maximal regularity is proved by \cite{AH} using their theory of diagonal dominated block operators.
 Different from our $A_{00}$, the linearized operator corresponding to compressible Navier-Stokes equations is a diagonal dominant operator in the sense of \cite{AH}.
\begin{rem} \label{RPe}
We consider the operator $A_{000}$ for the case $\Omega=\mathbb{T}^n$, where $A_{000}$ is defined in the beginning of the proof of Theorem~\ref{TSol}, namely
\[
	A_{000}U=\dbinom{\operatorname{div}u}{\nabla(\varepsilon B_0\rho)+L_0 u}
	\quad\text{for}\quad
	U=\dbinom{\rho}{u}.
\]
Since $A_{000}$ has a non-trivial kernel, argument above is not applicable.
 However, it is easy to see that $\ker A_{00}=\ker A_{000}=\mathbb{C}^n$ because $(\rho,u)\in\ker A_{00}$ or $(\rho,u)\in\ker A_{000}$ implies that $u$ is a constant by \eqref{EKI0} and Lemma~\ref{LLame}.
\end{rem}
\begin{proof}[Proof of Theorem~\ref{TRes} admitting Corollary~\ref{CMax}]
We set $R=A_0-A_{00}$ and observe
\[
	(\zeta + A_0)^{-1} = (\zeta + A_{00})^{-1}
	\left( I + R(\zeta+A_{00})^{-1} \right)^{-1}.
\]
Thus if we are able to prove that, the operator norm
\begin{equation} \label{ERi}
	\left\| R(\zeta+A_{00})^{-1} \right\|_{X_0\to X_0} \le 1,
\end{equation}
then by the Neumann series argument
\[
	\left( I + R(\zeta+A_{00})^{-1} \right)^{-1}
\]
is a well-defined operator in $\mathcal{L}(X_0,X_0)$.
 Thus $(\zeta+A_0)^{-1}$ is in $\mathcal{L}(X_0,X_1)$ if $(\zeta+ A_{00})^{-1}\in\mathcal{L}(X_0,X_1)$.

We shall prove \eqref{ERi} for sufficiently large $\zeta$.
 We know $R$ is a lower order term so that
\[
		\| RU \|_{X_0} \le \delta \| U \|_{X_1}
		+ C_\delta \| U \|_{X_0}
\]
for any $\delta>0$ with some constant $C_\delta>0$.
 For $A_{00}$, we know its sectoriality (Corollary \ref{CMax})
so that
\begin{gather*}
		\left\| (\zeta+A_{00})^{-1} U \right\|_{X_0} \le \frac{C_1}{\zeta} \| U \|_{X_0} \\
		\left\| A_{00}(\zeta+A_{00})^{-1} U \right\|_{X_0} \le C_2 \| U \|_{X_0}
\end{gather*}
for $\zeta>0$.
 Since $\|U\|_{X_1}\le C_3\|A_{00}U\|_{X_0}$, we now observe that
\[
	\left\| R(\zeta + A_{00})^{-1} U \right\|_{X_0} 
	\le \delta C_2 C_3 \| U \|_{X_0}
	+ \frac{C_\delta}{\zeta} C_1 \| U \|_{X_0}. 
\]
Take $\delta$ small so that $\delta C_2 C_3<1/2$ and take $\zeta$ large so that $C_\delta C_1/\zeta<1/2$.
 We have proved \eqref{ERi}.
\end{proof}

\section{Application of stability principle} \label{SA} 

We recall a generalized stability principle \cite{PWS} in the form \cite[Theorem~5.3.1]{PS}.
 Let us write their abstract setting.
 Let $X_0$ and $X_1$ be two complex Banach spaces such that $X_1$ is densely embedded in $X_0$.
 Let $(X_0,X_1)_{\theta,q}$ denote its real interpolation space \cite{L2}.
 We set
\[
	X_\gamma = (X_0,X_1)_{1-1/p,p}
\]
for a given $p\in(1,\infty)$.
 Let $V$ be an open set of $X_\gamma$.
 We consider
\[
	(\mathcal{A},\mathcal{F}): V 
	\to \mathcal{L}(X_1,X_0) \times X_0
\]
and a quasilinear evolution equation
\begin{equation} \label{EAbs}
	\frac{dv}{dt} + \mathcal{A}(v)v
	= \mathcal{F}(v), \quad
	v(0) = v_0 \in V.
\end{equation}
Let $v_*$ denote a \emph{equilibrium} of \eqref{EAbs}, i.e.,
\[
	\mathcal{A}(v_*) v_*
	= \mathcal{F} (v_*), \quad
	v_* \in V \cap X_1.
\]

For $\mathcal{A}$ at $v_*$ we assume
\begin{enumerate}
\item[(A1)] $\mathcal{A}(v_*)$ has the property of maximal $L^p$-regularity.
 In other words, a linear equation
\[
	\frac{dw}{dt} + \mathcal{A}(v_*)w
	= f \quad\text{in}\quad (0,T)
	\quad\text{with}\quad w(0) = w_0 \in X_\gamma
\]
is uniquely solvable for $f\in L^p(0,T,X_0)$ and
\[
	\int_0^T \left( \left\| \frac{dw}{dt} \right\|_{X_0}^p
	+ \| w \|_{X_1}^p \right) dt
	\le C_T \left( \int_0^T \| f \|_{X_0}^p\, dt 
	+ \| w_0 \|_{X_\gamma}^p \right)
\]
for any $T>0$ with some constant $C_T$ independent of $f$ and $w_0$.
\end{enumerate}
This makes the problem a parabolic problem.
 The term involving $\mathcal{F}$ is considered as a lower order term.
  When we study a large time behavior, we have to study a spectral property of the linearized operator of $\mathcal{A}(v)v-\mathcal{F}(v)$ at $v=v_*$.
 This operator is denoted by $\mathcal{A}_0$ and its explicit form is
\[
	\mathcal{A}_0 w =
	\mathcal{A} (v_*)w + \left( A'(v_*)w \right) v_*
	- \mathcal{F}'(v_*)w \quad\text{for}\quad w \in X_1.
\]
To define $\mathcal{A}_0$, we assume $C^1$-dependence of $\mathcal{A}$ and $\mathcal{F}$.
\begin{enumerate}
\item[(A2)] The mapping $(\mathcal{A},\mathcal{F}):V\to\mathcal{L}(X_1,X_0)\times X_0$ is $C^1$.
\end{enumerate}

We say that an equilibrium $v_*$ is (\emph{linearly}) \emph{normally stable} if the following four properties hold.
 Let $\mathcal{E}$ be the set of all equilibria in $V\cap X_1$.
\begin{enumerate}
\item[(S1)] Near $v_*$, the set $\mathcal{E}$ is $C^1$-manifold in $X_1$ of dimension $m\in\mathbb{N}$,
\item[(S2)] the tangent space for $\mathcal{E}$ at $v_*$ is isomorphic to $\ker\mathcal{A}_0$,
\item[(S3)] $0$ is a semi-simple eigenvalue of $\mathcal{A}_0$ in the sense that $\ker\mathcal{A}_0\oplus\operatorname{im}\mathcal{A}_0=X_0$,
 where $\ker\mathcal{A}_0$ denotes the nulty or the kernel of $\mathcal{A}_0$ and $\operatorname{im}\mathcal{A}_0$ denotes the range or the image of $\mathcal{A}_0$.
 (If we assume that $A_0$ is a Fredholm operator of index zero, i.e.,
\[
	\dim(\ker\mathcal{A}_0) - \dim(\operatorname{coker}\mathcal{A}_0) = 0, \quad
	\operatorname{coker}\mathcal{A}_0 = X_0/\operatorname{im}\mathcal{A}_0,
\] 
we only assume $\ker\mathcal{A}_0\cap\operatorname{im}\mathcal{A}_0=\{0\}$ since this implies $\ker\mathcal{A}_0\oplus \operatorname{im}\mathcal{A}_0=X_0$.
 In particular, if $X_0$ has a finite dimension, semi-simplicity is equivalent to $\ker\mathcal{A}_0\cap\operatorname{im}\mathcal{A}_0=\{0\}$.)
\item[(S4)] Let $\sigma(\mathcal{A}_0)$ denote the set of all spectra of $\mathcal{A}_0$.
 Then $\sigma(\mathcal{A}_0)\backslash\{0\}\subset\left\{z\in\mathbb{C}\mid \operatorname{Re} z>\omega\right\}$ with some $\omega>0$.
\end{enumerate}
In the case $m=0$, i.e., $v_*$ is an isolated equilibrium we say that $v_*$ is \emph{linearly stable} if $\ker\mathcal{A}_0=\{0\}$ and (S4) holds.

The next theorem is taken from \cite[Theorem~5.3.1]{PS} with slightly different wording.
\begin{thm} \label{TGSP}
Assume that $(\mathcal{A},\mathcal{F})$ satisfies (A2).
 Let $v_*$ be an equilibrium of \eqref{EAbs}.
 Assume (A1) for $\mathcal{A}(v_*)$.
 If $v_*$ is linearly normally stable ((S1), (S2), (S3), (S4)) then, $v_*$ is nonlinearly stable in the sense that there is $\delta>0$ such that if $\|v_0-v_*\|_{X_\gamma}<\delta$, then there exists a unique global-in-time solution $v$ of \eqref{EAbs} with $v(0)=v_0$, and moreover $v(t)$ converges to some $v_{**}\in\mathcal{E}$ as $t\to\infty$ in $X_\gamma$ at exponential rate.
 ($v_{**}$ may not be the same as $v_*$).
\end{thm}

If one applies the proof for the case that $\mathcal{E}=\{v_*\}$, we get usual stability principle.
\begin{cor} \label{CSP}
Assume that $(\mathcal{A},\mathcal{F})$ satisfies (A2).
 Let $v_*$ be an isolated equilibrium of \eqref{EAbs}.
 Assume that (A1) for $\mathcal{A}(v_*)$.
 If $v_*$ is linearly stable, then $v_*$ is nonlinearly stable in the sense that there is $\delta>0$ such that if $\|v_0-v_*\|_{X_\gamma}<\delta$, there exists a unique globally-in-time solution $v$ of \eqref{EAbs} with $v(0)=v_0$ and moreover $\|v(t)-v_*\|_{X_\gamma}\le ce^{-\omega t}$ as $t\to\infty$.
\end{cor}
For stability of isolated stationary solution, the assertion that linear stability implies nonlinear stability can be proved even for fully nonlinear equations by analytic semigroup theory (without appealing maximum regularity theory) in a different framework; see e.g.\ \cite[Chapter~9]{L1}.

We shall apply Theorem \ref{TGSP} and Corollary \ref{CSP} to get Theorem \ref{TM2} and Theorem \ref{TM1}.

We take function spaces $X_0$ and $X_1$ as in Section \ref{SSFO}, which are Hilbert spaces.
 We set $p=2$ and write $X_\gamma$ as $X_{1/2}$.
 We set operators $\mathcal{A}(U)=A(U)$, $\mathcal{F}(U)=F(U)$ defined in Section \ref{SSFO}.
 By Proposition \ref{PBa}, (A2) is fulfilled.
 Since we know the $L^2$-maximum regularity for $A_{00}$ in Lemma~\ref{LMax}, (A1) for $A_{00}$ is proved.
 If $v_*$ is isolated, it is shown that $\ker A_0=\{0\}$ in Lemma \ref{LSp0}.
 The spectral analysis for $A_0$ also yields (S4).
 Thus, Theorem \ref{TM1} is proved.
 To show Theorem \ref{TM2}, we have to further check semi-simplicity of $0$ eigenvalue.
%
\begin{thm} \label{TSSi}
Assume that $\mu>0$ and $2\mu+\lambda>0$.
 If $0$ is an eigenvalue of $A_0$, then it is semi-simple.
\end{thm}
\begin{proof}
We first prove that $\ker A_0\cap\operatorname{im}A_0=\{0\}$.
 Namely,
\[
	A_0(A_0 U) = 0 \quad\text{implies}\quad
	A_0 U = 0, \quad U \in X_0.
\]
Since $A_0(A_0 U) = 0$ implies that $A_0U\in\ker A_0$, by Lemma \ref{LSp0}, this implies that
\[
	A_0 U = \left\{ \dbinom{\rho_0}{0} \biggm|
	B(\rho^e)\rho = 0 \right\}.
\]
In other words,
\begin{align*}
	&\operatorname{div}(\rho^e u) = \rho_0 \\
	& \nabla B \rho + \frac{1}{\rho^e} L_0 u = 0 \\
	& B \rho_0 = 0
\end{align*}
where $B=B(\rho^e)$.
 By \eqref{EKI} with $\xi=0$, we observe that
\begin{equation*}
	\int_\Omega \bar{u} L_0 u\, dx
	= \int_\Omega \rho_0 (B\bar{\rho})\, dx
	= \int_\Omega (B\rho_0) \bar{\rho}\, dx = 0,
\end{equation*}
which yields $\nabla u=0$ by Lemma \ref{LLame}.
 This implies that $u=0$ by the boundary condition $u=0$.
 Since $\rho_0-\operatorname{div} (\rho^e\bar{u})=0$, we now obtain $\rho_0=0$.
 We now conclude that $A_0U=0$.

We next prove that the dimension of the coker of $A_0$ equals the dimension of the kernel of $A_0$, i.e.,
\[
	\dim(\operatorname{coker}A_0) = \dim(\ker A_0), \quad
	\operatorname{coker} A_0 = X_0/\operatorname{im}A_0.
\]
This is equivalent to saying that $A_0$ is a Fredholm operator of index zero.
 Since $(\zeta+A_0)^{-1}$ is compact if it exists, the Fredholm index for $\zeta+A_0$ is independent of $\zeta\in\mathbb{C}$.
 Thus the index of $A_0$ must be zero.
 Thus, the above equality holds.
 We also give an explicit proof for the reader's convenience.
 We consider
\begin{align*}
	&\operatorname{div}(\rho_* u) = r \quad\text{in}\quad \Omega \\
	&\frac1{\rho_*} L_0 u + \nabla\varpi = f \quad\text{in}\quad \Omega \\
	&u = 0 \quad\text{on}\quad \partial\Omega.
\end{align*}
We apply Lemma \ref{LFS} to solve the system.
 We argue in the same way to solve \eqref{EH1}, \eqref{EH2} with \eqref{EH3} as in  the proof of Theorem \ref{TSol} to get $(u,\varpi)\in H^2\times H^1$.
 For $B=B(\rho_*)$ since $B$ is self-adjoint, it is easy to see that $\ker B\cap\operatorname{im}B=\{0\}$.
 Moreover, $\operatorname{coker}B$ and $\ker B$ have the same dimension.
 Thus
\[
	\varpi = B \rho_1 + \rho_0, \quad
	\rho_0 \in \ker B, \quad
	\rho_1 \in D(B)
\]
by a suitable choice of $\rho_0$ and $\rho_1$.
 Thus
\begin{align*}
	&\operatorname{div}(\rho_* u) = r \\
	&\frac1{\rho_*} L_0 u + \nabla B\rho_1 = f - \nabla\rho_0.
\end{align*}
We notice that
\[
	\dim \left\{ \nabla \rho_0 \mid
	\rho_0 \in \ker B \right\}
	=\dim \ker B.
\]
Thus $\dim(\operatorname{coker}A_0)=\dim(\ker A_0)$.

Since we know that $\ker A_0\cap\operatorname{im}A_0=\{0\}$, this implies that
\[
	X_0 = \ker A_0 \oplus \operatorname{im}A_0.
\]
In other words, $0$ is semi-simple eigenvalue of $A_0$.
\end{proof}

\section*{Acknowledgements}
This work was done as a part of research activities of Social Cooperation Program ``Mathematical Science for Refrigerant Thermal Fluids" sponsored by Daikin Industries, Ltd.\  at the University of Tokyo. The authors are grateful to members of the Technology Innovation Center of Daikin Industries, Ltd.\ for showing several interesting phenomena related to phase transition with fruitful discussion which triggered this work.
 The work of the first author was partly supported by the Japan Society for the Promotion of Science (JSPS) through the grants KAKENHI: JP24K00531, JP24H00183 and by Arithmer Inc., Daikin Industries, Ltd.\ and Ebara Corporation through collaborative grants.
 The work of the third author was partly supported by JSPS through the grant JP22K13946.

\bigskip

\noindent
(Y.~Giga)
{\it Email address}: \href{mailto:labgiga@ms.u-tokyo.ac.jp}{\nolinkurl{labgiga@ms.u-tokyo.ac.jp}}\\
(N.~Kajiwara)
{\it Email address}: \href{mailto:kajiwara.naoto.p4@f.gifu-u.ac.jp}{\nolinkurl{kajiwara.naoto.p4@f.gifu-u.ac.jp}}\\
(K.~Tsuda)
{\it Email address}: \href{mailto:k-tsuda@ip.kyusan-u.ac.jp}{\nolinkurl{k-tsuda@ip.kyusan-u.ac.jp}}

\end{document}